\documentclass{amsart}
\usepackage{amsfonts,amssymb,amsmath,amsthm,amscd}
\usepackage{url}
\usepackage{enumerate}

\newtheorem{theorem}{Theorem}[section]
\newtheorem{lemma}[theorem]{Lemma}
\newtheorem{proposition}[theorem]{Proposition}
\newtheorem{corollary}[theorem]{Corollary}
\newtheorem{conjecture}[theorem]{Conjecture}
\newtheorem{hypothesis}[theorem]{Hypothesis}
\theoremstyle{definition}

\newenvironment{remark}[1][Remark]{\begin{trivlist}
\item[\hskip \labelsep {\bfseries #1}]}{\end{trivlist}}

\DeclareFontEncoding{OT2}{}{}
\newcommand{\textcyr}[1]{{\fontencoding{OT2}\fontfamily{wncyr}\fontseries{m}\fontshape{n}\selectfont #1}}

\newcommand{\Sha}{{\mbox{\textcyr{Sh}}}}
\newcommand{\ord}{{\operatorname{ord}}}
\newcommand{\Gal}{{\operatorname{Gal}}}
\newcommand{\Heeg}{{\operatorname{Heeg}}}
\newcommand{\cd}{{\operatorname{cd}}}
\newcommand{\Norm}{{\operatorname{Norm}}}
\newcommand{\cyc}{{\operatorname{cyc}}}
\newcommand{\coker}{{\operatorname{coker}}}
\newcommand{\Pic}{{\operatorname{Pic}}}

\newcommand{\rec}{{\operatorname{rec}}}

\newcommand{\tors}{{\operatorname{tors}}}
\newcommand{\cotors}{{\operatorname{cotors}}}
\newcommand{\im}{{\operatorname{im}}}

\newcommand{\Hom}{{\operatorname{Hom}}}
\newcommand{\Sel}{{\operatorname{Sel}}}

\newcommand{\GL}{{\operatorname{GL_2}}}
\newcommand{\ns}{{\operatorname{ns}}}

\author{Jeanine Van Order}
\address{Einstein Institute of Mathematics, The Hebrew University of Jerusalem}
\email{jeaninevanorder@gmail.com}
\thanks{The author acknowledges partial support from the Swiss National Science Foundation (FNS) grant 200021-125291.}
\keywords{Iwasawa theory, abelian varieties, elliptic curves}
\subjclass{Primary 11, Secondary 11G05, 11G10, 11G40, 11R23}
\begin{document}

\title[Dihedral Euler characteristics of Selmer groups of abelian varieties]
{On the Dihedral Euler characteristics of Selmer groups of abelian varieties}

\begin{abstract} This note shows how to use the framework of Euler characteristic formulae to study Selmer groups of 
abelian varieties in certain dihedral or anticyclotomic extensions of CM fields via Iwasawa main conjectures,
and in particular how to verify the $p$-part of the refined Birch and Swinnerton-Dyer conjecture in this setting. When the 
Selmer group is cotorsion with respect to the associated Iwasawa algebra, we obtain the $p$-part of formula predicted by 
the refined Birch and Swinnerton-Dyer conjecture. When the Selmer group is not cotorsion with respect to the associated 
Iwasawa algebra, we give a conjectural description of the Euler characteristic of the cotorsion submodule, and explain 
how to deduce inequalities from the associated main conjecture divisibilities of Perrin-Riou and Howard. \end{abstract}

\maketitle
\tableofcontents

\section{Introduction}

Fix a prime number $p$. Let $G$ be a compact $p$-adic Lie group without $p$-torsion. Writing $\mathcal{O}$ to denote the ring of integers of some fixed finite 
extension of ${\bf{Q}}_p$, let \begin{align*} \Lambda(G) &= \varprojlim_{U} \mathcal{O}[G/U]\end{align*} denote the $\mathcal{O}$-Iwasawa algebra of $G$. Here, 
the projective limit runs over all open normal subgroups $U$ of $G$, and $\mathcal{O}[G/U]$ denotes the usual group ring of $G/U$ with coefficients in $\mathcal{O}$. 
Let $M$ be any discrete, cofinitely generated $\Lambda(G)$-module. If the homology groups $H_i(G,M)$ are finite for all integers $i \geq 0$, then we say that the $G$-Euler characteristic $\chi(G, M)$ of $M$ is {\it{well-defined}}, and given by the (finite) value \begin{align*} \chi(G,M) &= \prod_{i\geq0} \vert H^i(G, M)\vert^{(-1)^i}. \end{align*} If $G$ 
is topologically isomorphic to ${\bf{Z}}_p^{\delta}$ for some integer $\delta \geq 1$, then a clearer picture of this Euler characteristic $\chi(G,M)$ emerges. In particular, a classical result shows that $\chi(G,M)$ is well defined if and only if $H_0(G,M)$ is finite. Moreover, if $M$ in this setting is a pseudonull $\Lambda(G)$-module, then $\chi(G, M)=1$. Now, the structure theory of finitely generated torsion $\Lambda(G)$-modules shown in Bourbaki \cite{BB} also gives the following Iwasawa theoretic picture of the $G$-Euler 
characteristic $\chi(G,M)$. That is, suppose $M$ is a discrete, cofinitely generated $\Lambda(G)$-cotorsion module. Let $M^{\vee} = \Hom(M, {\bf{Q}}_p/{\bf{Z}}_p)$ denote 
the Pontragin dual of $M$, which has the structure of a compact  $\Lambda(G)$-module. Then, there exists a finite collection of non-zero elements $f_1, \ldots, f_r \in \Lambda(G)$,
along with a pseudonull $\Lambda(G)$-module $D^{\vee}$, such that the following sequence is exact: \begin{align*} 0 \longrightarrow \bigoplus_{i=1}^r \Lambda(G)/f_i \Lambda(G) &\longrightarrow M^{\vee}\longrightarrow D^{\vee} \longrightarrow 0. \end{align*} The product \begin{align*} \operatorname{char}(M^{\vee}) &= \operatorname{char}_{\Lambda(G)}(M^{\vee}) = \prod_{i =1}^r f_i ,\end{align*} is unique up to unit in $\Lambda(G)$, and defines the {\it{$\Lambda(G)$-characteristic power series of $M^{\vee}$}}. If the homology group $H_0(G,M)$ is finite, then the $G$-Euler characteristics $\chi(G, M)$ and $\chi(G, D)$ are both well-defined, and given by the formulae \begin{align*}\chi(G, M) &= 
\vert \operatorname{char}(M^{\vee})(0) \vert_{p}^{-1} \\ \chi(G, D)  &=1.\end{align*} Here, $\operatorname{char}(M^{\vee})(0)$ denotes the image of the characteristic power 
series $\operatorname{char}(M^{\vee})$ under the natural map $\Lambda(G) \longrightarrow {\bf{Z}}_p$, and $D$ denotes the discrete dual of $D^{\vee}$.

In this note, we show how to use the framework of Euler characteristics to study the Selmer groups of abelian varieties in certain abelian extensions of CM 
fields, and in particular to verify the $p$-part of the refined Birch and Swinnerton-Dyer conjecture via the associated Iwasawa main conjectures. To be more precise, let $A$ be a principally polarized abelian variety defined over a totally real number field $F$. We shall assume for simplicity that $A$ is principally polarized, though all of the arguments given
below can be carried through to the more general case with some extra care. Fix a prime $\mathfrak{p}$ above $p$ in $F$. Let $K$ be a totally imaginary quadratic extension of $F$. Assume that $A$ has good ordinary reduction at each prime above $p$ in $K$. Let $K_{\mathfrak{p}^{\infty}}$ denote the dihedral or anticyclotomic ${\bf{Z}}_p^{\delta}$-extension of $K$, where $\delta = [F_{\mathfrak{p}}:{\bf{Q}}_p]$ is the residue degree. This $p$-adic Lie extension, whose construction we recall below, factors through the tower 
of ring class fields of $\mathfrak{p}$-power conductor over $K$. Let $G$ denote the Galois group $\Gal(K_{\mathfrak{p}^{\infty}}/K)$, which is topologically isomorphic to 
${\bf{Z}}_p^{\delta}$. Let $\Sel(A/K_{\mathfrak{p}^{\infty}})$ denote the $p^{\infty}$-Selmer group of $A$ over $K_{\mathfrak{p}^{\infty}}$. Hence by definition, 
\begin{align*}\Sel(A/K_{\mathfrak{p}^{\infty}}) &= \varinjlim_L \Sel(A/L),\end{align*} where the inductive limit ranges over all finite extensions $L$ of $K$ contained in 
$K_{\mathfrak{p}^{\infty}}$ of the classically defined Selmer groups \begin{align*} \Sel(A/L) &= \ker \left( H^1(G_S(L), A_{p^{\infty}}) \longrightarrow \bigoplus_{v \in S} J_v(L)\right).\end{align*} Here, $S$ is any (fixed) finite set of primes of $K$ that includes both the primes above $p$ and the primes where $A$ has bad reduction. We write $K^S$ to denote the maximal Galois extension of $K$ unramified outside of $S$ and the archimedean places of $K$, and $G_S(L)$ to denote the Galois group $\Gal(K^S/L)$. We write $A_{p^{\infty}}$ to denote the collection of all $p$-power torsion of $A$ over the relevant extension of $K$ in $K^S$, i.e. $A_{p^{\infty}} = \bigcup_{n \geq 0} A_{p^n}$ with 
$A_{p^n} = \ker ([p^n]: A \longrightarrow A)$, which is standard notation. In general however, given an abelian group $B$, we shall write $B(p)$ to denote its $p$-primary component. We define \begin{align*} J_v(L) &= \bigoplus_{w \mid v} H^1(L_w, A(\overline{L}_w))(p), \end{align*} where the sum runs over primes $w$ above $v$ in $L$. Note that each Selmer group $\Sel(A/L)$ fits into the short exact descent sequence \begin{align*} 0 \longrightarrow A(L) \otimes {\bf{Q}}_p/{\bf{Z}}_p \longrightarrow \Sel(A/L) \longrightarrow 
\Sha(A/L)(p) \longrightarrow 0, \end{align*} where $A(L)$ denotes the Mordell-Weil groupof $A$ over $L$, and $\Sha(A/L)(p)$ the $p$-primary part of the Tate-Shafarevich group 
$\Sha(A/L)$ of $A$ over $L$. Now, the $p$-primary of $p^{\infty}$-Selmer group $\Sel(A/K_{\mathfrak{p}^{\infty}})$ has the structure of a discrete, cofinitely generated 
$\Lambda(G)$-module. It is often (but not always) the case that $\Sel(A/K_{\mathfrak{p}^{\infty}})$ is $\Lambda(G)$-cotorsion. For instance, this is known by a standard argument
to be the case when the $p$-primary Selmer group $\Sel(A/K)$ is finite. Let us first consider this case where $\Sel(A/K_{\mathfrak{p}^{\infty}})$ is $\Lambda(G)$-cotorsion.
This setting is generally known as the {\it{definite case}} in the literature, for reasons that we do not dwell on here.\footnote{Roughly, when $A/F$ is a modular abelian variety, 
the associated Hasse-Weil $L$-function $L(A/K,s)$ has an analytic continuation a product of Rankin-Selberg $L$-functions of cuspidal eigenforms on some totally definite quaternion algebra over $F$, and this in particular can be used to deduce that the root number of $L(A/K,s)$ is equal to $1$ (as opposed to $-1$). We refer the reader to the 
e.g. works \cite{BD}, \cite{PW}, \cite{Nek}, or \cite{VO2} for some of the many known Iwasawa theoretic results in this direction.} In any case, the conjecture of Birch and Swinnerton-Dyer predicts that the Hasse-Weil $L$-function $L(A/K, s)$ has an analytic continuation to $s=1$, with order of vanishing at $s=1$ given by the rank of the Mordell-Weil 
group $A(K)$. The refined conjecture predicts that Tate-Shafarevich group $\Sha(A/K)$ is finite, and moreover that the leading coefficient of $L(A/K,s)$ at $s=1$ is given by the formula \begin{align}\label{RBSDT} \frac{\vert \Sha(A/K)\vert \cdot \operatorname{Reg}(A/K)\cdot \Phi(A/K)}{\vert A(K)_{\tors}\vert \cdot \vert A^t(K)_{\tors}\vert \cdot  \vert D \vert^{\frac{\dim(A)}{2}}} \cdot \prod_{v \mid \infty \atop v:K \rightarrow {\bf{R}}} \int_{A(K_v)} \vert \omega \vert \cdot \prod_{v \mid \infty \atop v: K \rightarrow {\bf{C}}}2 \int_{A(K_v)} \omega \wedge \overline{\omega}. \end{align} Here, $\operatorname{Reg}(A/K)$ denotes the regulator of $A$ over $K$, i.e. the determinant of the canonical height pairing on a basis of $A(K)/A(K)_{\tors}$, $\Phi(A/K)$ denotes the product \begin{align*} \Phi(A/K) &= \prod_{v \nmid \infty} c_v \left\vert \frac{\omega}{\omega_v^*}\right\vert_v, \end{align*} 
where $c_v = c_v(A)=[A(K_v):A_0(K_v)]$ is the local Tamagawa factor of $A$ at a finite prime $v$ of $K$, $\omega = \omega_A$ a fixed nonzero global exterior form (an invariant differential if $A$ is an elliptic curve), and $\omega_v^* = \omega_{A, v}^*$  the N\'eron differential at $v$. Moreover, $A^t$ denotes the dual abelian variety associated to $A$, and $D = D_K$ the absolute discriminant of $K$. We refer the reader to the article of Tate \cite{Ta} for details.\footnote{The reader will note that we have included extraneous real places in the formula $(\ref{RBSDT})$ above, i.e. as $K$ here is totally imaginary.} We consider a $p$-adic analogue of this formula $(\ref{RBSDT})$. Namely, let $\widetilde{A}(\kappa_v)$ denote the reduction of $A$ at a finite prime $v$ of $K$. Write $A(K)(p)$ to denote the $p$-primary component of $A(K)$, and $\widetilde{A}(\kappa_v)(p)$ that of $\widetilde{A}(\kappa_v)$. We first establish the following result, using standard techniques.

\begin{theorem}[Theorem \ref{main-definite}] 

Assume that $A$ is a principally polarized abelian variety 
having good ordinary reduction at each prime above $p$ in $K$. Assume additionally that the $p^{\infty}$-Selmer
group $\Sel(A/K)$ is finite. If the $p^{\infty}$-Selmer group $\Sel(A/K_{\mathfrak{p}^{\infty}})$ is $\Lambda(G)$-cotorsion, 
then the $G$-Euler characteristic $\chi(G, \Sel(A/K_{\mathfrak{p}^{\infty}}))$ is well defined, and given by the formula 
\begin{align}\label{definite-tamagawa} \chi(G,\Sel(A/K_{\mathfrak{p}^{\infty}})) = \frac{\vert  \Sha(A/K)(p) \vert}
{\vert A(K)(p)\vert^2} \cdot \prod_{v \mid p} \vert \widetilde{A}(\kappa_v)(p) \vert^2 \cdot \prod_{v \nmid p \atop \in S_{\ns}}
\ord_p \left(\frac{c_v(A)}{L_v(A/K,1)} \right). \end{align} Here, $S_{\ns}$ denotes the subset of $S$ of primes which
do not split completely in $K_{\mathfrak{p}^{\infty}}$, and $L_v(A/K, 1)$ denotes the Euler factor at $v$ of the global 
$L$-value $L(A/K, 1)$. \end{theorem} This result has various antecedents in the literature (see e.g. \cite{C}, \cite{CS}, or \cite{Ze}), though the setting 
of dihedral extensions considered here differs greatly from the cyclotomic or single-variable settings  
considered in these works. Anyhow, let us now examine the complementary setting where the $p^{\infty}$-Selmer group 
$\Sel(A/K_{\mathfrak{p}^{\infty}})$ is {\it{not}} expected to be $\Lambda(G)$-cotorsion. This setting is generally known as the {\it{indefinite case}} 
in the literature, for reasons that we likewise do not dwell on here.\footnote{Roughly (as before), in the setting where $A$ is a modular abelian variety 
defined over $F$, the Hasse-Weil $L$-function $L(A/K,s)$ has an analytic continuation given by the product of Rankin-Selberg
 $L$-functions of cuspidal eigenforms on some totally indefinite quaternion algebra over $F$ (where indefinite here means
 ramified at all but one real place of $F$), and this in particular can be used to deduce that the root number of $L(A/K,s)$ is equal 
 to $-1$ (as opposed to $1$). We refer the reader to the excellent paper of Howard \cite{Ho2} for an account of the Iwasawa theory
 in this setting.} To fix ideas, let us assume the so-called {\it{weak Heegner hypothesis}} that $\chi(N) = (-1)^{d-1}$, where $\chi$ denotes the 
 quadratic character associated to $K/F$, $N$ the arithmetic conductor of $A$, and $d =[F:{\bf{Q}}]$ the degree of $F$. This condition ensures
 that the abelian variety $A$ comes equipped with families of Heegner or CM points defined over ring class extensions of $K$. 
 We refer the reader to \cite{Ho2} or the discussion below for details. Let us first consider the issue of describing the $G$-Euler 
 characteristic of the $\Lambda(G)$-cotorsion submodule $\Sel(A/K_{\mathfrak{p}^{\infty}})_{\cotors}$ of $\Sel(A/K_{\mathfrak{p}^{\infty}})$, 
 following the conjecture of Perrin-Riou \cite{PR87} and Howard \cite{Ho2}. To state this version of the Iwasawa main conjecture, let 
 us write $\mathfrak{S}_{\infty} = \mathfrak{S}(A/K_{\mathfrak{p}^{\infty}})$ denote the compactified $p^{\infty}$-Selmer group of $A$ in 
 $K_{\mathfrak{p}^{\infty}}$, which has the structure of a compact $\Lambda(G)$-module. Let $\mathfrak{H}_{\infty}$ denote the 
 $\Lambda(G)$-submodule of $\mathfrak{S}_{\infty}$ generated by the images of Heegner points of $\mathfrak{p}$-power 
 conductor under the appropriate norm homomorphisms. We refer the reader to \cite{Ho} or to the passage below for more details of this construction. 
 It is often known and generally expected to be the case that the $\Lambda(G)$-modules $\mathfrak{S}_{\infty}$ and $\mathfrak{H}_{\infty}$ are 
 torsionfree of $\Lambda(G)$-rank one, in which case the quotient $\mathfrak{S}_{\infty}/\mathfrak{H}_{\infty}$ has $\Lambda(G)$-rank zero. Hence in this setting, the quotient 
 $\mathfrak{S}_{\infty}/\mathfrak{H}_{\infty}$ is a torsion $\Lambda(G)$-module, so has a well defined $\Lambda(G)$-characteristic power series 
 $\operatorname{char}(\mathfrak{S}_{\infty}/\mathfrak{H}_{\infty})$. Let $\operatorname{char}(\mathfrak{S}_{\infty}/\mathfrak{H}_{\infty})^*$ denote the image of 
 $\operatorname{char}(\mathfrak{S}_{\infty}/\mathfrak{H}_{\infty})$ under the involution of $\Lambda(G)$ induced by inversion in $G.$  Given an integer $n \geq 0$, let 
 $K[\mathfrak{p}^{n}]$ denote the ring class field of conductor $\mathfrak{p}^n$ over $K$, with Galois group $G[\mathfrak{p}^n]$. Write 
 $K[\mathfrak{p}^{\infty}] = \bigcup_{n \geq 0} K[\mathfrak{p}^n]$ to denote the union of all ring class extensions of $\mathfrak{p}$-power conductor, with Galois group 
$G[\mathfrak{p}^{\infty}] = \Gal(K[\mathfrak{p}^{\infty}]/K) = \varprojlim_n G[\mathfrak{p}^n]$. Recall that we let $K_{\mathfrak{p}^{\infty}}$ denote the dihedral or anticyclotomic 
${\bf{Z}}_p^{\delta}$-extension of $K$. We write $G = \varprojlim_n G_{\mathfrak{p}^n}$ to denote the associated profinite Galois group 
$G_{\mathfrak{p}^n} = \Gal(K_{\mathfrak{p}^n}/K)$. Let $\mathfrak{K}_n$ denote the Artin symbol  of $\mathfrak{d}_n = (\sqrt{D}\mathcal{O}_K) \cap \mathcal{O}_{p^n}$. 
Here, $D = D_K$ denotes the absolute discriminant of $K$, and $\mathcal{O}_{p^n}$ the $\mathcal{O}_F$-order of conductor $\mathfrak{p}^n$ in $K$, i.e. 
$\mathcal{O}_{p^n}= \mathcal{O}_F + \mathfrak{p}^n \mathcal{O}_K$. Let $\mathfrak{K} = \varprojlim_n \mathfrak{K}_n$. Let $\langle ~,~\rangle_{A, K[p^n]}$ denote the 
$p$-adic height pairing \begin{align*} \langle ~,~\rangle_{A, K[p^n]}: A^t(K[p^n]) \times A(K[p^n]) &\longrightarrow {\bf{Q}}_p, \end{align*} defined in Howard \cite[(9), $\S 3.3$]{Ho} 
and Perrin-Riou \cite{PR91}. Let us again assume for simplicity that $A$ is principally polarized. There exists by a construction of Perrin-Riou \cite{PR87} (see also \cite{PR91}) a 
$p$-adic height pairing \begin{align*} \mathfrak{h}_n: \mathfrak{S}(A/K_{p^n}) \times \mathfrak{S}(A/K_{p^n}) &\longrightarrow c^{-1}{\bf{Z}}_p \end{align*} whose restriction to the 
image of the Kummer map $A(K[p^n]) \otimes {\bf{Z}}_p \rightarrow \mathfrak{S}(E/K[p^n])$ coincides with the pairing $\langle ~,~\rangle_{A, K[p^n]}$ after identifying $A^t \cong A$ 
in the canonical way (cf. \cite[Proposition 0.0.4]{Ho}). Here, $c \in {\bf{Z}}_p$ is some integer that does not depend on the choice of $n$. Following Perrin-Riou \cite{PR87} and 
Howard \cite{Ho}, these pairings can be used to construct a pairing \begin{align*} \mathfrak{h}_{\infty}: \mathfrak{S}_{\infty} \times \mathfrak{S}_{\infty} &\longrightarrow 
c^{-1}{\bf{Z}}_p[[G]] \\ (\varprojlim a_n, \varprojlim_n b_n) &\longmapsto \varprojlim_n \sum_{\sigma \in G_{p^n}} \mathfrak{h}_n(a_n, b_n^{\sigma}) \cdot \sigma. \end{align*} 
We then define the {\it{$p$-adic regulator $\mathcal{R} = \mathcal{R}(A/K_{p^{\infty}})$}} to be the image of $\mathfrak{S}_{\infty}$ in $c^{-1}{\bf{Z}}_p[[G]]$ under this pairing 
$\mathfrak{h}_{\infty}$. Let ${\bf{e}}$ denote the natural projection \begin{align*} {\bf{e}}: {\bf{Z}}_p[[G[p^{\infty}]]] \longrightarrow {\bf{Z}}_p[[G]]. \end{align*} We then make the 
following conjecture, reformulating those of \cite{PR87}, \cite{Ho} and \cite{Ho2}.

\begin{conjecture}[Conjecture \ref{-1}] 

Let $A$ be a principally polarized abelian variety defined over $F$ having good ordinary reduction at each prime above $p$ in $K$. Assume that $A$ satisfies the 
weak-Heegner hypothesis with respect to $N$ and $K$. Then, the $G$-Euler characteristic $\chi(G, \Sel(A/K_{\mathfrak{p}^{\infty}})_{\cotors})$ is well-defined, and given 
by the formula \begin{align*} \chi(G, \Sel(A/K_{\mathfrak{p}^{\infty}})_{\cotors}) &= \vert \operatorname{char}(\mathfrak{S}_{\infty}/\mathfrak{H}_{\infty})(0) \cdot 
\operatorname{char}(\mathfrak{S}_{\infty}/\mathfrak{H}_{\infty})^*(0) \cdot \mathfrak{R}(0)\vert_p^{-1}. \end{align*} Here, $\mathfrak{R}$ denotes the ideal 
$({\bf{e}}(\mathfrak{K}))^{-1} \mathcal{R}$, which lies in the ${\bf{Z}}_p$-Iwasawa algebra ${\bf{Z}}_p[[G]]$.\end{conjecture} We refer the reader to the divisibility 
proved in Howard \cite[Theorem B]{Ho2} for results towards this conjecture. We can also deduce the following result toward this conjecture in the setting where $A$ is an 
elliptic curve defined over the totally real field ${\bf{Q}}$. Hence, $A$ is modular by the fundamental work of Wiles \cite{Wi}, Taylor-Wiles \cite{TW}, and 
Breuil-Conrad-Diamond-Taylor \cite{BCDT}. Let $f \in S_2(\Gamma_0(N))$ denote the cuspidal newform of level $N$ associated to $A$ by modularity. 
Let $K$ be an imaginary quadratic field in which the fixed prime $p$ os not ramified. Let $K(\mu_{p^{\infty}})$ denote the extension obtained from $K$ 
by adjoining all primitive $p$-power roots of unity, with $\Gamma = \Gal(K(\mu_{p^{\infty}})/K)$ its Galois group. Thus, $\Gamma$ is topologically isomorphic to 
${\bf{Z}}_p^{\times}$. Let $\mathcal{L}_f$ denote the two-variable $p$-adic $L$-function 
constructed by Hida \cite{Hi} and Perrin-Riou \cite{PR87}, which lies in the ${\bf{Z}}_p$-Iwasawa algebra ${\bf{Z}}_p[[G[p^{\infty}] \times \Gamma]]$, as explained in the proof of 
\cite[Theorem 2.9]{VO3}. Fixing a topological generator $\gamma$ of $\Gamma$, we write the expansion of $\mathcal{L}_f$ in the cyclotomic variable $\gamma -1$ as 
\begin{align} \label{cycexp} \mathcal{L}_f &= \mathcal{L}_{f,0} + \mathcal{L}_{f,1}(\gamma -1) + \mathcal{L}_{f,2}(\gamma -1)^2 + \ldots ,\end{align} with 
$\mathcal{L}_{f, i} \in {\bf{Z}}_p[[G[p^{\infty}]]]$ for each $i \geq 0$. If we assume the so-called {\it{Heegner hypothesis}} that each prime dividing $N$ splits in $K$, 
then root number of the Hasse-Weil $L$-function $L(A/K, s)$ equals $-1$, which forces the leading term $\mathcal{L}_{f,0}$ in $(\ref{cycexp})$ to vanish. 
Let $\Lambda(G)$ denote the ${\bf{Z}}_p$-Iwasawa algebra ${\bf{Z}}_p[[G]]$. We deduce the following result from the $\Lambda(G)$-adic Gross-Zagier 
theorem of Howard \cite[Theorem B]{Ho} with the two-variable main conjecture shown in Skinner-Urban \cite[Theorem 3]{SU} and the integrality of the 
two-variable $p$-adic $L$-function $\mathcal{L}_f$ shown in \cite{VO3}.
 
\begin{corollary}[Corollary \ref{EC-1}] 

Let $A$ be an elliptic curve defined over ${\bf{Q}}$, associated by modularity to a newform $f \in S_2(\Gamma_0(N))$. Let $K$ be an imaginary quadratic 
extension of ${\bf{Q}}$ of discriminant $D$ in which the fixed prime $p$ is not ramified. Assume that the elliptic curve $A$ has good ordinary reduction at $p$, 
that each prime dividing $N$ splits in $K$, and that $D$ is odd and not equal to $-3$. Assume in addition that the absolute Galois group 
$G_K =\Gal(\overline{K}/K)$ surjects onto the ${\bf{Z}}_p$-automorphisms of $T_p(A)$, that $\chi(p)=1$, that $p$ does not divide the class number of $K$,
and that the compactified Selmer group $\mathfrak{S}_{\infty}$ has $\Lambda(G)$-rank one. Then the $G$-Euler characteristic of $\chi(G, \Sel(A/K_{p^{\infty}})_{\cotors})$ 
is well-defined, and given by the formulae
\begin{align*} \chi(G, \Sel_{p^{\infty}}(A/K_{p^{\infty}})_{\cotors}) & \geq \vert {\bf{e}}(\mathcal{L}_{f,1}(0)) \vert_p^{-1} 
 \\ &= \vert  {\bf{e}}(\mathfrak{K} \cdot \log_p(\gamma))^{-1}\mathfrak{h}_{\infty}(\widetilde{x}_{\infty}, 
\widetilde{x}_{\infty})(0) \vert_p^{-1} \\ &= \left\vert \operatorname{char}(\mathfrak{S}_{\infty}/\mathfrak{H}_{\infty})(0) 
\cdot \operatorname{char}(\mathfrak{S}_{\infty}/\mathfrak{H}_{\infty})^*(0) \cdot \mathfrak{R}(0) \right \vert_p^{-1}.\end{align*} 
Here, $\log_p$ is the $p$-adic logarithm (composed with a fixed isomorphism $\Gamma \cong {\bf{Z}}_p^{\times}$), 
$\widetilde{x}_{\infty}$ is a generator of $\mathfrak{H}_{\infty}$ constructed from a compatible sequence of regularized Heegner
points (described below), and $\mathfrak{R}$ is the ideal $({\bf{e}}(\mathfrak{K} \cdot \log_p(\gamma))^{-1} \mathcal{R} 
\in \Lambda(G)$. Moreover, these formulae for $\chi(G, \Sel(A/K_{p^{\infty}})_{\cotors})$ do not depend on the choice of 
topological generator $\gamma \in \Gamma$. \end{corollary} Though relatively simple to deduce, this result is revealing in that it 
links the nonvanishing of the coefficient $\mathcal{L}_{f,1}$ to the nondegeneracy of the height pairing $\mathfrak{h}_{\infty}$ and the $p$-adic regulator
$\mathcal{R}$. Anyhow, we have shown here we can verify the $p$-part of the refined Birch and Swinnerton-Dyer formula (and variations) by relatively 
straightforward computations of Euler characteristics after knowing partial results towards the associated Iwasawa main conjectures. It would be interesting 
to relate these formula more precisely to the modular setting outlined in Pollack-Weston \cite{PW}, as well as perhaps to the conjectures of 
Shimura/Prasanna \cite{Pr} via Ribet-Takahashi \cite{RT}. It would also be interesting to relate the conjectural formula described above in the 
indefinite setting to the conjecture of Kolyvagin \cite{Ko} on indivisibility of Heegner points (and natural extensions to CM points on quaternionic 
Shimura curves). However, such problems lie beyond the scope of the present note.

\section{Ring class towers}

\begin{remark}[The dihedral or anticyclotomic ${\bf{Z}}_p^{\delta}$-extension $K_{\mathfrak{p}^{\infty}}$ of $K$.]

Suppose that $F$ is any totally real number field, and $K$ any totally imaginary quadratic extension of $F$. Recall that we fix throughout a rational prime $p$, 
and let $\mathfrak{p}$ denote a fixed prime above $p$ in $F$. Given an integral ideal $\mathfrak{c} \subseteq \mathcal{O}_F$, we write 
$\mathcal{O}_{\mathfrak{c}} = \mathcal{O}_F + \mathfrak{c}\mathcal{O}_K$ to denote the $\mathcal{O}_F$-order of conductor $\mathfrak{c}$ in $K$. 
The {\it{ring class field of conductor $\mathfrak{c}$ of $K$}} is the Galois extension $K[\mathfrak{c}]$ of $K$ characterized via class field theory 
by the identification \begin{align*}\begin{CD} \widehat{K}^{\times}/\widehat{F}^{\times} \widehat{\mathcal{O}}_{\mathfrak{c}}^{\times}K^{\times} 
@>{\rec_K}>> \Gal(K[\mathfrak{c}]/K). \end{CD}\end{align*} Here, $\rec_K$ denotes the Artin reciprocity map, normalized to send uniformizers to 
geometric Frobenius automorphisms. Let us write $G[\mathfrak{c}]$ to denote the Galois group $\Gal(K[\mathfrak{c}]/K)$. We consider the union
of all ring class extensions of $\mathfrak{p}$-power conductor \begin{align*} K[\mathfrak{p}^{\infty}] = \bigcup_{n \geq 0} K[\mathfrak{p}^n], \end{align*} along 
with its Galois group $G[\mathfrak{p}^{\infty}]= \Gal(K[\mathfrak{p}^{\infty}]/K)$, whose profinite structure we write as \begin{align*}  G[\mathfrak{p}^{\infty}] 
&= \varprojlim_n G[\mathfrak{p}^n]. \end{align*} A standard argument in the theory of profinite groups shows that the torsion subgroup 
$G[\mathfrak{p}^{\infty}]_{\tors} \subseteq G[\mathfrak{p}^{\infty}]$ is finite, and moreover that the quotient $G[\mathfrak{p}^{\infty}]/G[\mathfrak{p}^{\infty}]_{\tors}$ 
is topologically isomorphic to ${\bf{Z}}_p^{\delta}$, where $\delta = [F_{\mathfrak{p}}:{\bf{Q}}_p]$ is the residue degree of $\mathfrak{p}$. We refer the reader e.g. to 
\cite[Corollary 2.2]{CV} for details on how to deduce this fact via a basic computation with adelic quotient groups. Let $G_{\mathfrak{p}^{\infty}}$ denote the Galois 
group $G[\mathfrak{p}^{\infty}]/G[\mathfrak{p}^{\infty}]_{\tors} \cong {\bf{Z}}_p^{\delta}$, which has the structure of a profinite group \begin{align*} G_{\mathfrak{p}^{\infty}} 
&= \varprojlim_n G_{\mathfrak{p}^n}.\end{align*} Let us then write $K_{\mathfrak{p}^n}$ to denote the abelian extension of $K$ contained in $K[\mathfrak{p}^n]$ such that 
$G_{\mathfrak{p}^n} = \Gal(K_{\mathfrak{p}^n}/K)$. When there is no risk of confusion, we shall simply write $G$ to denote the Galois group $G_{\mathfrak{p}^{\infty}}.$ 

\begin{lemma}\label{ram} The extension $K_{\mathfrak{p}^{\infty}}$ over $K$ is unramified outside of $p$. \end{lemma}

\begin{proof} The result is a standard deduction in local class field theory, as shown for instance in Iwasawa \cite[Theorem 1 $\S 2.2$]{Iw}. In general, given any $p$-adic field 
$L$, the abelianization of $G_L = \Gal(\overline{L}/L)$ is a completion of $L^{\times}$. The subgroup corresponding to ramified extensions is $\mathcal{O}_L^{\times},$ which 
is a profinite abelian group whose open subgroups are pro-$p$. The result is then easy to deduce. \end{proof} Finally, let us recall that the Galois extension 
$K[\mathfrak{p}^{\infty}]$ is of generalized dihedral type. Equivalently, the complex conjugation automorphism of $\Gal(K/F)$ acts by inversion on $G[\mathfrak{p}^{\infty}]$. 
It follows that $K[\mathfrak{p}^{\infty}]$ is linearly disjoint over $K$ to the cyclotomic extension $K(\mu_{p^{\infty}})$, where $\mu_{p^{\infty}}$ denotes the set of all $p$-power 
roots of unity, because the complex conjugation automorphism acts trivially on $\Gal(K(\mu_{p^{\infty}})/K)$. For this reason, the extension $K_{\mathfrak{p}^{\infty}}/K$ is often 
called the {\it{anticyclotomic}} ${\bf{Z}}_p^{\delta}$-extension of $K$. \end{remark}

\begin{remark}[Decomposition of primes in $K_{\mathfrak{p}^{\infty}}$.]

Let us now write $G$ to denote the Galois group $G_{\mathfrak{p}^{\infty}} = \Gal(K_{\mathfrak{p}^{\infty}}/K)$. 
Given a finite prime $v$ of $K$, let $D_v \subseteq G$ denote the decomposition subgroup at a fixed prime 
above $v$ in $K_{\mathfrak{p}^{\infty}}.$ We shall often make the standard
identification $D_v$ with the Galois group $G_w=\Gal(K_{\mathfrak{p}^{\infty}, w}/K_v),$
where $w$ is a fixed prime above $v$ in $K_{\mathfrak{p}^{\infty}}.$

\begin{proposition}\label{decomposition} Let $v$ be a finite prime of $K$.

\begin{itemize}

\item[(i)] If $v$ does not divide $p$ and splits completely in $K_{\mathfrak{p}^{\infty}}$, then $D_v$ is trivial.

\item[(ii)] If $v$ does not divide $p$ and does not split completely in $K_{\mathfrak{p}^{\infty}}$, then $D_v$ 
can be identified with a finite-index subgroup of $G \approx {\bf{Z}}_p^{\delta}$. 

\item[(iii)] If $v$ divides $p$, then $D_v$ is topologically isomorphic to ${\bf{Z}}_p^{\delta}$. \end{itemize}
\end{proposition} 

\begin{proof} The result is easy to deduce from standard properties of local fields, see e.g. \cite[Proposition 5.10]{Cox}. \end{proof} \end{remark}

\section{Galois cohomology}

We now record for later use some basic facts from Galois cohomology.

\begin{remark}[$p$-cohomological dimension.] 

Recall that the $p$-cohomological dimension $\cd_p(G)$ of a profinite group $G$ is the smallest integer $n$ satisfying
the condition that for any discrete torsion $G$-module $M$ and any integer $q>n$, the $p$-primary component of $H^q(G, M)$ 
is zero (\cite[$\S 3.1$]{Se}). We have the following crucial characterization of $\cd_p$.

\begin{theorem}[Serre-Lazard]\label{SL} 

Let $G$ be any $p$-adic Lie group. Let  $M$ be any discrete $G$-module. If $G$ does not contain an element of order $p$, then $\cd_p(G)$ 
is equal to the dimension of $G$ as a $p$-adic Lie group.\end{theorem}

\begin{proof} See \cite{Se2} and \cite{Laz}. \end{proof} Since $G_{\mathfrak{p}^{\infty}} \cong {\bf{Z}}_p^{\delta}$ has no point of order $p$, we obtain the following consequence.

\begin{corollary}\label{decomposition2} 

Let $G$ denote the Galois group $G_{\mathfrak{p}^{\infty}} \cong {\bf{Z}}_{p}^{\delta}$ defined above. Then, $\cd_p(G) = \delta$. Moreover, given a finite prime $v$ of $K$,

\begin{itemize}

\item[(i)]$\cd_p(D_v) = 0$ if $v$ does not divide $p$ and splits completely in $K_{\mathfrak{p}^{\infty}}$;
\item[(ii)] $\cd_p(D_v) = \delta$ if $v$ does divide $p$.

\end{itemize} \end{corollary}

\begin{proof} The claim follows from Theorem \ref{SL}, using Proposition \ref{decomposition} for (i) and (ii). 
\end{proof} \end{remark}

\begin{remark}[The Hochschild-Serre spectral sequence.]

Fix $S$ any finite set of primes of $K$ containing 
the primes above $p$ in $K$, and the primes where $A$ has 
bad reduction. Let $K_S$ denote the maximal extension 
of $K$ that is unramified outside of $S$ and the archimedean 
places of $K$. Let $G_S(K) = \Gal(K_S/K) $ denote the corresponding 
Galois group. Observe that there is an inclusion of fields
$K_{\mathfrak{p}^{\infty}} \subset K_S$. Given any intermediate field 
$L$ with $K \subset L \subset K_{\mathfrak{p}^{\infty}}$, let \begin{align*} G_S(L) = \Gal(K_S/L).\end{align*} 
Recall again that we write $G$ to denote the Galois group $G_{\mathfrak{p}^{\infty}}
= \Gal(K_{\mathfrak{p}^{\infty}}/K)$.

\begin{lemma} For all integers $i \geq 1$, we have
bijections \begin{align}\label{HSSSg} H^i(G,
H^i(G_S(K_{\mathfrak{p}^{\infty}}), A_{p^{\infty}})&\cong
H^{i+2}(G, A_{p^{\infty}}). \end{align}
\end{lemma}

\begin{proof} See \cite[$\S$ 2.2.6]{Se} or \cite{Se2}. We claim that 
for all $i\geq 0$, a direct application of the Hochschild-Serre 
spectral sequence gives the exact sequence \begin{align*} H^{i+1}(G_S(K), 
A_{p^{\infty}}) \longrightarrow H^i(G, H^1(G_S(K_{\mathfrak{p}^{\infty}}), 
A_{p^{\infty}}))\longrightarrow H^{i+2}(G, A_{p^{\infty}}).\end{align*} 
On the other hand, it is well known that $H^i(G_S(K), A_{p^{\infty}}) 
= 0$ for all $i \geq 2$. Hence, the claim holds for all $i \geq 1$. \end{proof} \end{remark}

\begin{remark}[Some identifications in local cohomology.]

Let $L$ be a finite extension of $K$. Recall that for each finite prime $v$ of $K$, we define 
\begin{align}\label{J_v(L)} J_v(L) = \bigoplus_{w \mid v}  H^1(L_w, A(\overline{L}_w))(p), \end{align} where the sum runs over all primes $w$ above $v$ in $L$. 
If $L_{\infty}$ is an infinite extension of $K$, then we define the associated group $J_v(L_{\infty})$ by taking the inductive limit over finite extensions $L$ of $K$ 
contained in $L_{\infty}$ with respect to restriction maps, \begin{align}\label{J_v(H)} J_v(L_{\infty}) = \varinjlim_L J_v(L).\end{align} Given a finite prime 
$v$ of $K$, let us fix a prime $w$ above $v$ in $K_{\mathfrak{p}^{\infty}}$. Recall that we then write $G_w$ to denote the Galois group $\Gal(K_{\mathfrak{p}^{\infty}, w}/K_v)$, 
which we identify with the decomposition subgroup $D_v \subseteq G$.

\begin{lemma} Let $v$ be a finite prime of $K$, and $w$ any prime above $v$ in $K_{\mathfrak{p}^{\infty}}$. Then, for each integer $i \geq 0$, there is a canonical bijection
\begin{align}\label{shapiro} H^i(G, J_v(K_{\mathfrak{p}^{\infty}})) &\cong H^i(G_w, H^1(K_{\mathfrak{p}^{\infty}, w}, A)(p)). \end{align} \end{lemma}

\begin{proof} Recall that we let $G_{\mathfrak{p}^n}$ denote the Galois group $\Gal(K_{\mathfrak{p}^n}/K)$, which is isomorphic to $\left( {\bf{Z}}/p^n {\bf{Z}}\right)^{\delta}.$ 
Let $G_{\mathfrak{p}^n, w} \subseteq G$ denote the decomposition subgroup of the restriction of $w$ to $K_{\mathfrak{p}^n}$. Shapiro's lemma implies that for all integers 
$n\geq 1$ and $i\geq 0$, we have canonical bijections \begin{align*} H^i(G_{\mathfrak{p}^n}, J_v(K_{\mathfrak{p}^n})) &\cong H^i(G_{\mathfrak{p}^n, w}, 
H^1(K_{\mathfrak{p}^n, w}, A)(p)).\end{align*}  Passing to the inductive limit with $n$ then proves the claim. \end{proof}

\begin{corollary}\label{vnmidpvanish} Let $v$ be any finite prime of $K$ not dividing $p$ which splits completely in $K_{\mathfrak{p}^{\infty}}$, and let $w$ be any prime above 
$v$ in $K_{\mathfrak{p}^{\infty}}$. Then, for each integer $i \geq 1$, we have that $H^i(G_w, J_v(K_{\mathfrak{p}^{\infty}})) = 0$. \end{corollary}

\begin{proof} The result follows from $(\ref{shapiro})$, since $G_w$ is trivial by Proposition \ref{decomposition}. \end{proof}

\begin{lemma} Let $v$ be any finite prime of $K$ that divides $p$, and write $w$ to denote the prime above $v$ in $K_{\mathfrak{p}^{\infty}}$. Then, for each $i \geq \delta -1$, 
we have that \begin{align*} H^i(G_w, H^1(K_{\mathfrak{p}^{\infty},w}, A)(p)) =0.\end{align*} Moreover, for each $1 \leq i \leq \delta -2 $, we have bijections \begin{align}\label{lcgHSSS} H^i(G_w, H^1(K_{\mathfrak{p}^{\infty},w}, A)(p)) &\cong H^{i+2}(G_w, \widetilde{A}_{v, p^{\infty}}).\end{align} \end{lemma}

\begin{proof} See \cite[Lemma 3.5]{C}.  Let $G_{K_v}$ denote the Galois group $\Gal(\overline{K}_v/K_v)$, where $\overline{K}_v$ is a fixed algebraic closure of $K_v$. 
Let $I_{K_v} \subseteq G_{K_v}$ denote the inertia subgroup at a fixed prime above $v$ in $\overline{K}_v$. Consider the short exact sequence of $G_{K_v}$-modules 
\begin{align}\label{62} 0 \longrightarrow C \longrightarrow A_{p^{\infty}} \longrightarrow D \longrightarrow 0.\end{align} Here, $C$ denotes the canonical subgroup associated 
to the formal group of the N\'eron model of $A$ over $\mathcal{O}_{K_v}$ (see \cite[$\S$ 4]{CG}). Moreover, $(\ref{62})$ is characterized by the fact that $C$ is divisible, with 
$D$ being the maximal quotient of $A_{p^{\infty}}$ by a divisible subgroup such that the inertia subgroup $I_{K_v}$ acts via a finite quotient. Since $K_{\mathfrak{p}^{\infty}, w}$ 
is deeply ramified in the sense of Coates-Greenberg \cite{CG}, we have by \cite[Propositions 4.3 and 4.8]{CG} a canonical $G_w$-isomorphism 
\begin{align}\label{63} H^1(K_{\mathfrak{p}^{\infty}, w}, A)(p) &\cong H^1(K_{\mathfrak{p}^{\infty}, w}, D).\end{align} Moreover, the 
$G_{K_v}$-module $D$ vanishes if and only if $A$ has potentially supersingular reduction at $v$, as explained in \cite[$\S$ 4]{CG}.

Observe now that $\Gal(\overline{K}_v/K_{\mathfrak{p}^{\infty}, w})$ has $p$-cohomological dimension equal to $1$, as the profinite degree of $K_{\mathfrak{p}^{\infty}, w}$ 
over $K_v$ is divisible by $p^{\infty}$ (see \cite{Se}). Hence, for all $i \geq 2$,  we have identifications 
\begin{align}\label{64} H^i(K_{\mathfrak{p}^{\infty},w}, D) = 0.\end{align} Taking $(\ref{64})$ along with the Hochschild-Serre spectral sequence then gives us for each 
$i \geq 1$ exact sequences \begin{align}\label{65} H^{i+1}(K_v, D) \longrightarrow H^i(G_w, H^1(K_{\mathfrak{p}^{\infty}, w}, D)) \longrightarrow H^{i+2}(G_w, D). \end{align} 
While we would like to follow the argument of \cite[Lemma 3.5]{C} in showing that the terms on either side of $(\ref{65})$ must vanish, we have not been able to find an argument 
to deal with the fact that $\delta$ could be greater than $4$. Anyhow, we can deduce that \begin{align}\label{66} H^i(K_v, D)=0\end{align} for all $i \geq 2$. That is, we know that 
$\cd_p\left( G_{K_v} \right) = 2$, whence $(\ref{66})$ holds trivially for all $i \geq 3$. To deduce the vanishing for $i=2$, we take cohomology of the exact sequence 
$(\ref{62})$ to obtain a surjection from $H^2(K_v, A_{p^{\infty}})$ onto $H^2(K_v, D)$. We know by local Tate duality that $H^2(K_v, A_{p^{\infty}})$ 
vanishes (see e.g. \cite[1.12 Lemma]{CS}). Hence, we deduce that $(\ref{66})$ holds for all $i \geq 2$. It then follows from the exact sequence $(\ref{65})$ that we 
have bijections \begin{align*} H^i(G_w, H^1(K_{\mathfrak{p}^{\infty},w}, D)) \longrightarrow H^{i+2}(G_w, D) \end{align*} for all $i \geq 1$. Now, observe that both of these 
groups vanish for all  $i \geq \delta -1 $, as $\cd_p(G_w) = \delta$. The result now follows from the natural identification of $D$ with the $G_w$-module defined by
$\widetilde{A}_{p^{\infty}, v} = A(K_{\mathfrak{p}^{\infty}, w})_{p^{\infty}}$. \end{proof} \end{remark}

\section{The torsion subgroup $A(K_{\mathfrak{p}^{\infty}})_{\tors}$}

We start with the following basic result, whose proof relies on the fact that there exist primes $v \subset \mathcal{O}_K$ which split completely 
in the anticyclotomic extension $K_{\mathfrak{p}^{\infty}}$.

\begin{lemma}\label{finitetorsion}

The torsion subgroup $A(K_{\mathfrak{p}^{\infty}})_{\tors}$ of $A(K_{\mathfrak{p}^{\infty}})$ is finite. \end{lemma}

\begin{proof} 

Fix a finite prime $v$ of $F$ that remains inert in $K$. Let us also write $v$ to denote the prime above $v$ in $K$. By class field theory, $v \subseteq \mathcal{O}_K$ splits 
completely in $K_{\mathfrak{p}^{\infty}}$. Hence, fixing a prime $w$ above $v$ in $K_{\mathfrak{p}^{\infty}}$, we can identify the union of all completions 
$K_{\mathfrak{p}^{\infty}, w}$ with $K_v$ itself. In particular, we find that $A(K_{\mathfrak{p}^{\infty}, w}) = A(K_v)$, which gives us an injection $A(K_{\mathfrak{p}^{\infty}})_{\tors} 
\longrightarrow A(K_v)_{\tors}.$ On the other hand, consider the short exact sequence \begin{align}\label{sses}0 \longrightarrow A_1(K_v) \longrightarrow A_0(K_v)
\longrightarrow \widetilde{A}(\kappa_v) \longrightarrow 0. \end{align} Here, $A_1(K_v)$ denotes  the kernel of reduction modulo $v$, which can be identified with the formal 
group $\widehat{A}(\mathfrak{m}_v)$, and $\kappa_v$ denotes the residue field of $K$ at $v$. Let $n \geq 1$ be any integer prime to the residue characteristic $q_v$. It is 
well-known that $\widehat{A}(\mathfrak{m}_v)[n]=0$. Hence, taking $n$-torsion in $(\ref{sses})$, we find for any such prime $v \subset \mathcal{O}_K$ an injection \begin{align}\label{ti} A_0(K_v)[n] \longrightarrow \widetilde{A}(\kappa_v). \end{align} The group $\widetilde{A}(\kappa_v)$ is of course finite. Since we have this injection $(\ref{ti})$ for any such prime $v \subset \mathcal{O}_K$, it follows that $A(K_{\mathfrak{p}^{\infty}})_{\tors}$ must be finite. \end{proof} 

We now consider the $G$-Euler characteristic of the $p$-primary subgroup $A_{p^{\infty}}$ of $A(K_{\mathfrak{p}^{\infty}})$, denoted here by $\chi(G, A_{p^{\infty}})$. Recall that this is defined by the product \begin{align*} \chi(G, A_{p^{\infty}}) = \prod_{i \geq 0} \vert H^i(G, A_{p^{\infty}}) \vert^{(-1)^i}, \end{align*} granted that it is well-defined.

\begin{corollary}\label{chiA=1} We have that $\chi(G, A_{p^{\infty}})$ is well-defined and equal to $1$. \end{corollary}

\begin{proof} This is a standard result, since $G$ is a pro-$p$ group with $A_{p^{\infty}}(K_{\mathfrak{p}^{\infty}})$ a finite group of order equal to some power of $p$. See \cite[Exercise (a) $\S 4.1$]{Se}. \end{proof}  Recall that given a finite prime $v$ if $K$, we let $G_w$ denote the Galois group $\Gal(K_{\mathfrak{p}^{\infty}, w}/K_v)$, 
where $w$ is a fixed prime above $v$ in $K_{\mathfrak{p}^{\infty}}$. Let $\widetilde{A}_{v, p^{\infty}}$ denote the $G_w$-module defined by $A_{p^{\infty}}(K_{\mathfrak{p}^{\infty},w})$. Consider the $G_w$-Euler characteristic of $\widetilde{A}_{p^{\infty}, v}$, which we denote by $\chi(G_w, \widetilde{A}_{v, p^{\infty}})$. Recall that this is defined by the product \begin{align*} \chi(G_w, \widetilde{A}_{p^{\infty}, v}) = \prod_{i \geq 0} \vert H^i(G_w, \widetilde{A}_{v, p^{\infty}})\vert^{(-1)^i},
\end{align*} granted that it is well-defined.

\begin{corollary}\label{chiA_p=1}

Let $v$ be any finite prime of $K$, with $w$ a fixed prime above $v$ in $K_{\mathfrak{p}^{\infty}}.$ The $G_w$-Euler characteristic 
$\chi(G_w, A_{\mathfrak{p}^{\infty},v})$ is well-defined and equal to $1$.

\end{corollary}

\begin{proof} In any case on the decomposition of $v$ in $K_{\mathfrak{p}^{\infty}}$, $G_w$ is pro-$p$, with $\widetilde{A}_{p^{\infty}, v}$ a finite group of order equal to some 
power of $p$, whence the result is standard. Note as well that since $G_w \cong D_v$ is in any case a closed subgroup of $G$, the result can also be deduced from 
Corollary \ref{chiA=1} above, since Shapiro's lemma gives canonical isomorphisms $H^i(G, A_{p^{\infty}}) \cong H^i(G_w, \widetilde{A}_{v, p^{\infty}})$ for each 
$i \geq 1$. \end{proof} Putting this all together, we have shown the following result.

\begin{proposition}\label{chi} Let $v$ be any finite prime of $K$, with $w$ a fixed prime above $K$ in $K_{\mathfrak{p}^{\infty}}$. Then, 
$ \chi(G, A_{p^{\infty}}) = \chi(G_w, \widetilde{A}_{v, p^{\infty}}) = 1.$ \end{proposition}

\section{The definite case}

Let us keep all of the notations of the section above. Here, we shall assume that the $p^{\infty}$-Selmer group $\Sel(A/K_{\mathfrak{p}^{\infty}})$ is $\Lambda(G)$-cotorsion, 
where $\Lambda(G)$ denotes the $\mathcal{O}$-Iwasawa algebra $\Lambda(G_{\mathfrak{p}^{\infty}}) = \mathcal{O}[[G_{\mathfrak{p}^{\infty}}]]$ for $\mathcal{O}$ the ring of 
integers of some fixed finite extension of ${\bf{Q}}_p$. We shall also use here the standard notations for abelian varieties and local fields, following Coates-Greenberg \cite{CG}. 

\begin{remark}[Strategy.] 

We consider the snake lemma in the following commutative diagram

\begin{align}\label{fd}\begin{CD} \begin{CD} 
\Sel(A/K_{\mathfrak{p}^{\infty}})^{G} @>>> H^1(G_S(K_{\mathfrak{p}^{\infty}}), A_{p^{\infty}})^{G} 
@>{\psi_S(K_{\mathfrak{p}^{\infty}})}>> \bigoplus_{v \in S}J_v(K_{\mathfrak{p}^{\infty}})^{G} \\
     @AA{\alpha}A              @AA{\beta}A
@AA{\gamma = \bigoplus_{v\in S} \gamma_{v} }A     @. \\
\Sel(A/K) @>>> H^1(G_S(K), A_{p^{\infty}})
@>{\lambda_S(K)}>> \bigoplus_{v \in S}J_v(K).\\ \end{CD} \end{CD} \end{align} 
Here, recall that $S$ denotes any finite set of places of $K$ containing the primes above $p$, as well as all places where $A$ has bad reduction. The map 
$\psi_S(K_{\mathfrak{p}^{\infty}})$ is induced from the localization map 

\begin{align*} \lambda_S(K_{\mathfrak{p}^{\infty}}): H^1(G_S(K_{\mathfrak{p}^{\infty}}), A_{p^{\infty}}) &\longrightarrow \bigoplus_{s \in S} J_v(K_{\mathfrak{p}^{\infty}})\end{align*}
defining the $p^{\infty}$-Selmer group $\Sel(A/K_{\mathfrak{p}^{\infty}})$. The vertical arrows are induced by restriction on cohomology. We also consider the associated diagram 
\begin{align}\label{ifd}\begin{CD}\begin{CD} \im \left( \psi_S(K_{\mathfrak{p}^{\infty}}) \right) @>>> \bigoplus_{v \in S} J_v(K_{\mathfrak{p}^{\infty}}) @>>> \coker \left( \psi_S(K_{\mathfrak{p}^{\infty}})\right) \\
     @AA{\epsilon_1}A              @AA{\gamma = \bigoplus_{v\in S} \gamma_v}A
@AA{\epsilon_2}A     @. \\
\im \left( \lambda_S(K) \right) @>>>  \bigoplus_{v \in S}J_v(K) @>>> \coker \left( \lambda_S(K) \right).\\ \end{CD}\end{CD} \end{align} \end{remark}

\begin{remark}[Local restrictions maps.] 

We now compute the cardinalities of the kernel and cokernel of the restriction map $\gamma = \bigoplus_{v \in S} \gamma_v$ in $(\ref{fd})$, in a series of lemmas leading to 
Theorem \ref{LA} below. Throughout this section, we shall assume for simplicity that the abelian variety is principally polarized. 

\begin{lemma}\label{3.6} 

Let $v$ be a finite prime of $K$ which does not divide $p$. Assume that $A$ is principally polarized. Then, the group $J_v(K) = H^1(K_v, A)(p)$ is finite, and its order is given by 
the exact power of $p$ dividing the quotient $c_v(A)/L_v(A, 1)$. Here, $c_v(A) = [A(K_v): A_0(K_v)]$ denotes the local Tamagawa factor at $v$, and $L_v(A, 1)$ the Euler factor 
at $v$  of the Hasse-Weil $L$-function $L(A, s)$ at $s=1$. \end{lemma}

\begin{proof} The result is standard, see e.g. \cite[Lemma 3.6]{C}. \end{proof}

\begin{lemma}\label{3.7} 

Let $v$ be a finite prime of $K$ which does not divide $p$. Let $L$ be an arbitrary Galois extension of $K_v$ with Galois group $\Omega = \Gal(L/K_v)$. 
Let \begin{align*} \tau_v: H^1(K_v, A)(p) &\longrightarrow H^1(L, A)(p)^{\Omega}\end{align*} denote the natural map induced by restriction. Then, we have bijections 
\begin{align*} \ker(\tau_v) &\cong H^1(\Omega, A_{p^{\infty}}(L))\\ \coker(\tau_v) &\cong H^2(\Omega, A_{p^{\infty}}(L)). \end{align*}\end{lemma}

\begin{proof} See \cite[Lemma 3.7]{C}. That is, consider the commutative diagram \begin{align}\label{C37} \begin{CD} H^1(K_v, A_{p^{\infty}}) @>{s_v}>>
H^1(L, A_{p^{\infty}})^{\Omega}\\@VVV@VVV\\H^1(K_v, A)(p) @>{\tau_v}>> H^1(L, A)(p)^{\Omega}.\end{CD} \end{align} Here, $s_v$ denotes the restriction 
map on cohomology. The vertical arrows denote the surjective maps derived from Kummer theory on $A$. Now, observe that since $v$ does not divide $p$, 
we have identifications \begin{align*} A(K_v) \otimes {\bf{Q}}_p/{\bf{Z}}_p = A(L) \otimes {\bf{Q}}_p/{\bf{Z}}_p = 0. \end{align*} Hence, the vertical arrows in $(\ref{C37})$ 
are isomorphisms, and we have bijections \begin{align*} \ker(\tau_v) &\cong \ker(s_v)\\ \coker(\tau_v) &\cong \coker(s_v). \end{align*} The result now follows from the 
Hochschild-Serre spectral sequence associated to $s_v$, using the fact that $H^2(K_v, A_{p^{\infty}})=0$ by local Tate duality. \end{proof}

\begin{corollary}\label{LTF}

Let $v$ be a prime of $S$ not dividing $p$. If $v$ splits completely $K_{\mathfrak{p}^{\infty}}$, then the local restriction map 
$\gamma_v$ is an isomorphism, i.e. $\ker(\gamma_v) = \coker(\gamma_v) = 0$. If $v$ does not split completely in $K_{\mathfrak{p}^{\infty}}$, then $\gamma_v$
is the zero map, and the order of its kernel $J_v(K) = H^1(K_v, A)(p)$ is the exact power of $p$ dividing $c_v(A)/L_v(A,1)$. \end{corollary}

\begin{proof} The first assertion follows from the result of Corollary \ref{decomposition2}, which implies that $\cd_p(G_w) =0$ for any such prime $v$ of $K$ not dividing $p$.
To see the second assertion, it can be argued in the same manner as \cite[Lemma 3.3]{C} that the $J_v(K_{\mathfrak{p}^{\infty}})$ vanishes in this setting (i.e. since the Galois 
group of the extension $K_{\mathfrak{p}^{\infty}, w}$ over $K_v$ is isomorphic to a finite number of copies of the Galois group of the maximal unramified pro-$p$ extension of $K_v$). The result then follows from that of Lemma \ref{3.6}. Note as well that if we take $L= K_{\mathfrak{p}^{\infty}, w}$ with $\Omega = G_w$ in Lemma \ref{3.7}, we obtain identifications \begin{align*} \ker(\gamma_v) &= H^1(G_w, A_{p^{\infty}}) \\ \coker(\gamma_v) &= H^2(G_w, A_{p^{\infty}}), \end{align*} whence we can also deduce that 
$\vert H^1(G_w, A_{p^{\infty}}) \vert = \ord_p(c_v(A)/L_v(A,1))$ and that $H^2(G_w, A_{p^{\infty}}) = 0$. \end{proof} To study the localization maps $\gamma_v$ 
with $v$ dividing $p$ in the manner of \cite{C}, we start with the following basic result.

\begin{lemma}\label{0} Let $v$ be a prime of $K$ dividing $p$. Then, we have identifications \begin{align*}\ker(\gamma_v) &\cong H^1(G_w, A(K_{\mathfrak{p}^{\infty}, w}))(p) \\ 
\coker(\gamma_v) &\cong H^2(G_w, A(K_{\mathfrak{p}^{\infty}, w}))(p). \end{align*}\end{lemma}

\begin{proof} Taking the Hochschild-Serre spectral sequence associated to $K_{\mathfrak{p}^{\infty},w}$ over $K_v$ with the fact that $H^2(K_v, A_{p^{\infty}})=0$,
the result follows from the definition of $\gamma_v$. \end{proof} We now consider the primes $v \mid p$ of $K$ where $A$ has good ordinary reduction.

\begin{lemma}\label{I} Let $v$ be a prime of $K$ dividing $p$ where $A$ has good ordinary reduction. Assume that $A$ is principally polarized. Then, we have identifications

\begin{align*} \vert H^i(K_v,\widehat{A}_v(\overline{\mathfrak{m}}))\vert &= \begin{cases} 0 &\text{if $i \geq 2$}\\ \vert \widetilde{A}(\kappa_v)(p)\vert &\text{if $i=1.$}\end{cases}
\end{align*} \end{lemma}

\begin{proof} See \cite[Lemma 3.13]{C}, the proof carries over without changes to the setting of principally polarized abelian varieties. \end{proof}

\begin{lemma}\label{00} Let $v$ be a prime of $K$ above $p$ where $A$ has good ordinary reduction. Assume additionally that $A$ is principally polarized. 
Then, for all $i \geq 2$, we have isomorphisms \begin{align*} H^i(G_w, A(K_{\mathfrak{p}^{\infty}, w}))(p) &\cong H^i(G_w, \widetilde{A}_{v, p^{\infty}}).\end{align*} 
We also have for $i=1$ the short exact sequence \begin{align}\label{II} H^1(K_v, \widehat{A}_v(\overline{\mathfrak{m}})) \longrightarrow 
H^1(G_w, A(K_{\mathfrak{p}^{\infty},w}))(p) \longrightarrow H^1(G_w, \widetilde{A}_{v, p^{\infty}}). \end{align} \end{lemma}

\begin{proof} The results are deduced from the lemmas above, following \cite[Lemmas 3.12 and 3.14]{C}, which work in the same way for principally polarized abelian 
varieties using the results of \cite{CG}. \end{proof} Putting these identifications together, we obtain the following result. Let us write $h_{v, i}$ to denote the cardinality of 
$H^i(G_w, \widetilde{A}_{v, p^{\infty}})$ for any integer $i \geq 0$. 

\begin{proposition} Let $v$ be a prime of $S$ above $p$ where $A$ has good ordinary reduction. Assume additionally that $A$ is principally polarized. Then, $\ker(\gamma_v)$ 
and $\coker(\gamma_v)$ are finite, and we have the identification \begin{align*} \frac{\vert \ker(\gamma_v)\vert}{\vert \coker(\gamma_v) \vert} &= \frac{h_{v,0} \cdot h_{v, 1}}{h_{v, 2}}.
\end{align*}\end{proposition}

\begin{proof} Lemmas \ref{0} and \ref{00} imply that $\vert \coker(\gamma_v)\vert = h_{v,2}$. Using the exact sequence $(\ref{II})$ along with Lemma \ref{I} for $i=0$, we then deduce 
that $\vert \coker(\gamma_v) \vert = h_{v, 0} \cdot h_{v, 1}$, as required. \end{proof} 

To summarize what we have shown in this section, let us for a prime $v \in S$ not dividing $p$ write $\mathfrak{m}_v(A)$ to denote the exponent $\ord_p(c_v(A)/L_v(A, 1))$.

\begin{theorem}\label{LA} Let $S_{\ns}$ denote the subset of primes of the fixed set $S$ which (i) do not dividing $p$ and (ii) do not split completely in $K_{\mathfrak{p}^{\infty}}$.
We have the following description of the global restriction map $\gamma$ in $(\ref{fd})$: \begin{align*}\frac{\vert \ker(\gamma)\vert}{\vert \coker(\gamma) \vert} 
&= \prod_{v \mid p \atop v \in S} \frac{h_{v,0} \cdot h_{v, 1}}{h_{v, 2}} \times \prod_{v \nmid p \atop v \in S_{\ns}} \mathfrak{m}_v(A). \end{align*}\end{theorem} \end{remark}

\begin{remark}[Global calculation.] 

We start with the easy part of the calculation, which in some sense is just diagram chasing in $(\ref{fd})$. Let us now assume for simplicity that $A$ has good ordinary reduction at all primes above $p$ in $K$. 

\begin{lemma}\label{beta} We have the following bijections for the restriction map $\beta$: \begin{align*} \ker(\beta) &\cong H^1(G,A_{p^{\infty}}) \\ \coker(\beta) &\cong H^2(G,
A_{p^{\infty}}).\end{align*} \end{lemma}

\begin{proof} This follows immediately from the inflation-restriction exact sequence, using the fact that $H^2(G_S(K), A_{p^{\infty}})=0$. \end{proof} Recall that given a prime 
$v \in S$ and an integer $i \geq 0$, we write $h_{v,i}$ to denote the cardinality of $H^i(G_w, \widetilde{A}_{v, p^{\infty}})$. Let us also write $h_i$ to denote the cardinality of 
$H^i(G, A_{p^{\infty}})$. We now give the first part of the global calculation:

\begin{lemma} Assume that $\Sel(A/K)$ and $\coker(\lambda_S(K))$ are finite. Assume that $A$ has good ordinary reduction at each prime above $p$ in $K$. Then,
the cardinality $\vert H^0(G, \Sel(A/K_{\mathfrak{p}^{\infty}})) \vert $ is given by the formula

\begin{align}\label{global1} \vert \Sel(A/K) \vert \cdot \vert \coker(\psi_S(K_{\mathfrak{p}^{\infty}})) \vert \cdot \frac{h_2}{h_1\cdot h_0} \cdot \prod_{v \mid p} \frac{h_{v,0} \cdot h_{v,1}}{h_{v,2}} \cdot \prod_{v \nmid p \atop v \in S_{\ns}} \mathfrak{m}_v(A). \end{align} \end{lemma}

\begin{proof} Theorem \ref{LA} implies that \begin{align}\label{gamma} \frac{\vert \ker(\gamma)\vert}{\vert \coker(\gamma)\vert} = \prod_{v \mid p} \frac{h_{v,0} 
\cdot h_{v, 1}}{h_{v, 2}} \cdot \prod_{v \nmid p \atop v \in S_{\ns}} \mathfrak{m}_v(A).\end{align} Observe that the numerator and denominator of $(\ref{gamma})$ are finite 
by our assumption that $A$ has good ordinary reduction at each prime above $p$ in $K$. Using this result in the snake lemma associated to $(\ref{ifd})$, we deduce that 
$\ker(\epsilon_1)$ and $\coker(\epsilon_1)$ must be finite. Moreover, we deduce that \begin{align}\label{id1} \frac{\vert \ker(\epsilon_1)\vert}{\vert \coker(\epsilon_1)\vert} 
&= \frac{\vert \ker(\gamma)\vert }{\vert \coker(\gamma)\vert} \cdot \frac{\vert \coker\left( \psi_S(K_{\mathfrak{p}^{\infty}})\right) \vert}{\vert \coker\left( \lambda_S(K)\right) \vert}. \end{align} 
Now, it is well-known that if $\Sel(A/K)$ and $\coker(\lambda_S(K))$ are finite, then \begin{align}\label{id2} \vert \coker\left( \lambda_S(K) \right) \vert = A(K)_{p^{\infty}}= 
h_0.\end{align} On the other hand, taking the snake lemma in $(\ref{fd})$, we find that \begin{align}\label{id3} \frac{\vert \Sel(A/K_{\mathfrak{p}^{\infty}})^G\vert}{\vert \Sel(A/K) \vert} 
= \frac{\vert \coker(\beta)\vert}{\vert \ker(\beta)\vert} \cdot \frac{\vert \ker(\epsilon_1)\vert}{\vert \coker(\epsilon_1)\vert}. \end{align} The result now follows from $(\ref{id3})$ by 
$(\ref{gamma})$, $(\ref{id1})$ and $(\ref{id2})$. \end{proof}

\begin{corollary}\label{mainEC+1} If the localization map $\lambda_S(K_{\mathfrak{p}^{\infty}})$ is surjective, then \begin{align}\label{globalEC+1} \chi(G, \Sel(A/K_{\mathfrak{p}^{\infty}})) = \frac{\vert \Sel(A/K)\vert }{h_0^2} \cdot \prod_{v \mid p} h_{0,v}^2 \cdot \prod_{v \nmid p \atop v \in S_{\ns}} \mathfrak{m}_v(A). \end{align}\end{corollary}

\begin{proof} If the localization map $\lambda_S(K_{\mathfrak{p}^{\infty}})$ is surjective, then we have the short exact sequence
\begin{align}\label{locseq} 0 \longrightarrow \Sel(A/K_{\mathfrak{p}^{\infty}}) \longrightarrow B_{\infty} \longrightarrow C_{\infty}
\longrightarrow 0,\end{align} where \begin{align*} B_{\infty} &= H^1(G_S(K_{\mathfrak{p}^{\infty}}), A_{p^{\infty}}) \\ C_{\infty} &= \bigoplus_{v \in S}
J_v(K_{\mathfrak{p}^{\infty}}).\end{align*} Taking $G$-cohomology of $(\ref{locseq})$ then gives the long exact sequence 

\begin{align}\label{Gcoh}\begin{CD} 0 @>>> H^0(G, \Sel(A/K_{\mathfrak{p}^{\infty}})) @>>> H^0(G, B_{\infty}) @>{\psi_S(K_{\mathfrak{p}^{\infty}})}>> H^0(G, C_{\infty}) \\
@>>> H^1(G, \Sel(A/K_{\mathfrak{p}^{\infty}})) @>>> H^1(G, B_{\infty}) @>>> H^1(G, C_{\infty}) \\ @>>> \ldots \\ @>>> H^{\delta}(G, \Sel(A/K_{\mathfrak{p}^{\infty}})) 
@>>> H^{\delta}(G, B_{\infty}) @>>> H^{\delta}(G, C_{\infty})\\ @>>> 0.\end{CD}\end{align} Let us now write $$W_{i,v} = H^i(G_w, H^1(K_{\mathfrak{p}^{\infty}, v},
A_{p^{\infty}})(p)).$$ We can then extract from $(\ref{Gcoh})$ the long exact sequence \begin{align}\label{Gcoh2}\begin{CD} 0 @>>> H^0(G, \Sel(A/K_{\mathfrak{p}^{\infty}}))  
@>>> H^0(G, B_{\infty}) @>{\psi_S(K_{\mathfrak{p}^{\infty}})}>> \bigoplus_{v \in S} W_{0,v} \\ @>>> H^1(G, \Sel(A/K_{\mathfrak{p}^{\infty}})) @>>> H^3(G, A_{p^{\infty}}) @>>> \bigoplus_{v \in S} H^3(G_w, \widetilde{A}_{v, p^{\infty}}) \\ @>>> \ldots \\ @>>> H^{\delta}(G, \Sel(A/K_{\mathfrak{p}^{\infty}})) @>>> 0.\end{CD}\end{align} Here, 
the identifications for the central terms come from the Hochschild-Serre spectral sequence, $(\ref{HSSSg})$. The identifications of the terms on the right come from Corollary
\ref{vnmidpvanish} for the first row, then from the version of Shapiro's lemma given in $(\ref{shapiro})$ along with the bijections $(\ref{lcgHSSS})$ for the subsequent rows. Note that we have ignored the vanishing of higher cohomology groups here. The reason for this will me made clear below. That is, using the definition of the map 
$\psi_S(K_{\mathfrak{p}^{\infty}})$ in $(\ref{fd})$, we can then extract from $(\ref{Gcoh2})$ the exact sequence \begin{align}\label{extractedES}\begin{CD} 0 
@>>> \coker\left( \psi_S(K_{\mathfrak{p}^{\infty}}) \right) @>>> H^1(G, \Sel(A/K_{\mathfrak{p}^{\infty}})) @>>> H^3(G, A_{p^{\infty}})\\ 
@>>> \cdots @>>> H^{\delta}(G, \Sel(A/K_{\mathfrak{p}^{\infty}})) @>>> 0. \end{CD}\end{align} Using $(\ref{extractedES})$, we deduce that

\begin{align}\label{ECid2} \chi(G, \Sel(A/K_{\mathfrak{p}^{\infty}})) = \frac{\vert H^1(G, \Sel(A/K_{\mathfrak{p}^{\infty}}))\vert }{\vert \coker\left( \psi_S(K_{\mathfrak{p}^{\infty}})\right)\vert} \cdot \chi^{(3)}(G, A_{p^{\infty}}) \cdot \prod_{v \in S} \chi^{(3)}(G_w, \widetilde{A}_{v, p^{\infty}})^{-1}. \end{align} Here, we have written 

\begin{align*} \chi^{(3)}(G, A_{p^{\infty}}) = \prod_{i \geq 3} \vert H^i(G, A_{p^{\infty}}) \vert^{(-1)^i},\end{align*} and 

\begin{align*} \chi^{(3)}(G_w, A_{p^{\infty}}) = \prod_{i \geq 3} \vert H^i(G_w, \widetilde{A}_{v, p^{\infty}}) \vert^{(-1)^i}. \end{align*} Substituting $(\ref{global1})$ into $(\ref{ECid2})$, 
we then find that $\chi(G, \Sel(A/K_{\mathfrak{p}^{\infty}}))$ is equal to \begin{align*} \vert \Sel(A/K) \vert &\cdot \frac{h_2}{h_0 h_1} \cdot \prod_{v \mid p} \frac{h_{v, 0}h_{v,1}}{h_{v,2}} \cdot \prod_{v \nmid p \atop v \in S_{\ns}} \mathfrak{m}_v(A) \cdot \chi^{(3)}(G, A_{p^{\infty}}) \cdot \prod_{v \mid p} \chi^{(3)}(G_w, \widetilde{A}_{v, p^{\infty}})^{-1} \\ 
&= \vert \Sel(A/K) \vert \cdot \frac{\chi(G, A_{p^{\infty}})}{h_0^2} \cdot \prod_{v \nmid p \atop v \in S_{\ns}} \mathfrak{m}_v(A)  \cdot \prod_{v \mid p} h_{0,v}^2 \cdot \chi(G_w, \widetilde{A}_{v,p^{\infty}})^{-1}. \end{align*} The result then follows from Proposition \ref{chi}, which gives for each prime $v \mid p$ the identifications $\chi(G, A_{p^{\infty}})= \chi(G_w, \widetilde{A}_{v, p^{\infty}}) = 1$. \end{proof} \end{remark}

\begin{remark}[Surjectivity of the localization map.]

Recall that we write $\Lambda$ to denote the $\mathcal{O}$-Iwasawa algebra
$\Lambda(G) = \mathcal{O}[[G_{\mathfrak{p}^{\infty}}]]$. We now show 
that the localization map $\lambda_S(K_{\mathfrak{p}^{\infty}})$ 
is surjective if $\Sel_{p^{\infty}}(A/K_{\mathfrak{p}^{\infty}})$ is
$\Lambda$-cotorsion, at least when $A$ is known to be prinicipally
polarized. Given $L$ a Galois extension of $K$, let $\mathfrak{S}(A/L)$
denote the compactified Selmer group of $A$ over $L$, 
\begin{align*}\mathfrak{S}(A/L)&= \varprojlim_n \ker \left( H^1(G_S(L), A_{p^n}) 
\longrightarrow \bigoplus_{v \in S} J_v(L)\right). \end{align*}

\begin{proposition}\label{injAV} Let $\Omega = \Gal(L/K)$ be
an infinite, pro-$p$ group. If $A(L)_{p^{\infty}}$ is finite, then
there is a $\Lambda(\Omega)$-module injection \begin{align*}
\mathfrak{S}(A/L) \longrightarrow \Hom_{\Lambda(\Omega)}\left(
X(A/L), \Lambda(\Omega) \right).\end{align*}
\end{proposition}

\begin{proof}
The result is well-known, see for instance \cite[Theorem 7.1]{HV},
the proof of which carries over without change. \end{proof}

\begin{theorem}\emph{(Cassels-Poitou-Tate exact sequence for abelian varieties).}
Let $A^t$ denote the dual abelian variety associated to $A$. Let 
$\mathfrak{S}(A/K_{\mathfrak{p}^{\infty}})^{\vee}$ denote the Pontryagin dual
of the compactified Selmer group $\mathfrak{S}(A/K_{\mathfrak{p}^{\infty}})$.
There is a canonical exact sequence \begin{align}\label{CPTAV}
\begin{CD} 0 @>>> \Sel(A/K_{\mathfrak{p}^{\infty}})
@>>> H^1(G_S(K_{\mathfrak{p}^{\infty}}), A_{p^{\infty}}) @>>>
\bigoplus_{v\in S} H^1(K_{\mathfrak{p}^{\infty}, w}, A)(p) \\
@>>> \mathfrak{S}^{\vee}(A^t/K_{\mathfrak{p}^{\infty}}) @>>>
H^2(G_S(K_{\mathfrak{p}^{\infty}}), A_{p^{\infty}}) @>>> 0.
\end{CD}\end{align} \end{theorem}

\begin{proof} See Coates-Sujatha \cite[$\S$ 1.7]{CS}. The same deduction
using the generalized Cassels-Poitou-Tate exact sequence
\cite[Theorem 1.5]{CS} works here, keeping track of the dual abelian
variety $A^t$. \end{proof}

\begin{corollary}\label{surjlocAV} Assume that $A$ is principally
polarized. If $\Sel(A/K_{\mathfrak{p}^{\infty}})$ is $\Lambda$-cotorsion, 
then the localization map $\lambda_S(K_{\mathfrak{p}^{\infty}})$ is
surjective, i.e. there is a short exact sequence
\begin{align*}0 \longrightarrow \Sel(A/K_{\mathfrak{p}^{\infty}})
\longrightarrow  H^1(G_S(K_{\mathfrak{p}^{\infty}}), A_{p^{\infty}}) \longrightarrow 
\bigoplus_{v \in S} J_v(K_{\mathfrak{p}^{\infty}}) \longrightarrow 0.\end{align*}
\end{corollary}

\begin{proof}

If $\Sel(A/K_{\mathfrak{p}^{\infty}})$ is $\Lambda$-cotorsion, then
$X(A/K_{\mathfrak{p}^{\infty}})$ is $\Lambda$-torsion by duality. It
follows that $\Hom_{\Lambda}(X(A/K_{\mathfrak{p}^{\infty}}),\Lambda)
= 0$. Hence by Proposition \ref{injAV}, we find that $\mathfrak{S}(A/K_{\mathfrak{p}^{\infty}})=0$. 
It follows that $\mathfrak{S}^{\vee}(A/K_{\mathfrak{p}^{\infty}})=0$. Since
we assume that $A$ is principally polarized, we can assume
that there is an isomorphism $A \cong A^t$, in which case
it also follows that $\mathfrak{S}^{\vee}(A^t/K_{\mathfrak{p}^{\infty}})=0$.
The claim then follows from $(\ref{CPTAV})$. Observe that we also
obtain from this the vanishing of $H^2(G_S(K_{\mathfrak{p}^{\infty}}), A_{p^{\infty}})$. 
\end{proof} Hence, we obtain from Corollary \ref{mainEC+1} the following
main result of this section.

\begin{theorem}\label{main-definite} 

Assume that $A$ is a principally polarized, and that $A$ has good ordinary reduction at each prime above $p$ in $K$. Assume additionally that $\Sel(A/K)$
is finite. If $\Sel_{p^{\infty}}(A/K_{\mathfrak{p}^{\infty}})$ is $\Lambda(G)$-cotorsion, then 

\begin{align*} \chi(G,\Sel(A/K_{\mathfrak{p}^{\infty}})) = \frac{\vert  \Sha(A/K)(p) \vert}{\vert A(K)(p)\vert^2} \cdot \prod_{v \mid p} 
\vert \widetilde{A}(\kappa_v)(p) \vert^2 \cdot \prod_{v \nmid p \atop v \in S_{\ns}} \ord_p\left( \frac{c_v(A)}{L_v(A,1)}\right). \end{align*}\end{theorem}

\begin{proof} The computations above show that the Euler characteristic is well-defined, and given by the formula 
\begin{align*} \chi(G, \Sel(A/K_{\mathfrak{p}^{\infty}})) = \frac{\vert \Sel(A/K)\vert}{h_0^2} \cdot \prod_{v \mid p} h_{0,v}^2 
\cdot \prod_{v \nmid p \atop v \in S_{\ns}} \mathfrak{m}_v(A). \end{align*} The result then follows from the definitions, as well as 
the standard identification in this setting of $\Sel(A/K)$ with $\Sha(A/K)[p^{\infty}]$. \end{proof} \end{remark}

\section{The indefinite case}

Let us keep all of the setup above, but with the following crucial condition.
Let $N \subset \mathcal{O}_F$ denote the arithmetic conductor of $A$. Recall
that we write $\eta$ to denote the quadratic character associated to $K/F$. We 
assume here the that so-called {\it{weak Heegner hypothesis}} holds,
which is that $\eta(N) = (-1)^{d-1}$ (cf. \cite[0.]{Ho2}). In this setting, the Selmer
group $\Sel(A/K_{\mathfrak{p}^{\infty}})$ is generally not $\Lambda$-cotorsion, but 
rather free of rank one over $\Lambda$. For results in this direction, see Howard 
\cite[Theorem B]{Ho2}, which generalizes earlier results of Bertolini \cite{Be} and 
Nekovar (cf. \cite{Nek08} or \cite{Nek09}).

\begin{remark}[The conjectures of Perrin-Riou and Howard.]  

We refer the reader to \cite{Ho2} or \cite{PR87} for background. The version of the Iwasawa main conjectures posed in these works has the following interpretation in terms 
of Euler characteristic formulae. Recall that we fix a prime $\mathfrak{p}$ above $p$ in $F$. For each finite extension $L$ of $K$, we consider the two $\mathfrak{p}$-primary 
Selmer groups $\Sel_{\mathfrak{p}^{\infty}}(A/L)$ and $\mathfrak{S}(A/L)$, which can be defined implicitly via their inclusion in the descent exact sequences 
\begin{align*} 0 \rightarrow A(L) \otimes_{\mathcal{O}_F} \mathcal{O}_{F_{\mathfrak{p}}} 
\rightarrow \mathfrak{S}(A/L) \rightarrow \varprojlim_n \Sha(A/L)[\mathfrak{p}^n] \\
0 \rightarrow A(L) \otimes_{\mathcal{O}_F} \left( F_{\mathfrak{p}}/\mathcal{O}_{F_{\mathfrak{p}}} \right)
\rightarrow \Sel_{\mathfrak{p}^{\infty}}(A/L) \rightarrow \Sha(A/L)[\mathfrak{p}^{\infty}] \rightarrow 0.
\end{align*} The abelian variety $A$ comes equipped with a family of Heegner (or CM) points defined 
over the ring class extensions $K[\mathfrak{p}^n]$, as described e.g. in \cite[$\S 1.2$]{Ho2}, which gives
rise to the following submodule. Given an integer $n\geq 0$, let $h_n$ denote the image under the norm from 
$K[\mathfrak{p}^{n+1}]$ to $K_{\mathfrak{p}^n}$ of the Heegner/CM point of conductor $\mathfrak{p}^{n+1}$.
Let $\mathfrak{H}_n$ denote the $\Lambda$-module generated by all the images $h_m$ with $m \leq n$. Let 
$\mathfrak{H}_{\infty}= \varprojlim_n \mathfrak{H}_n$ denote the associated {\it{Heegner module}}. 
Taking the usual limits, let us also define the finitely-generated $\Lambda$-modules 
\begin{align*}\mathfrak{S}_{\infty} = \mathfrak{S}(A/K_{\mathfrak{p}^{\infty}}) &= \varprojlim_n 
\mathfrak{S}(A/K_{\mathfrak{p}^n})  \\ \Sel_{\mathfrak{p}^{\infty}}(A/K_{\mathfrak{p}^{\infty}}) 
&= \varinjlim_n \Sel_{\mathfrak{p}^{\infty}}(A/K_{\mathfrak{p}^n}).\end{align*} Let $X(A/K_{\mathfrak{p}^{\infty}})$ denote the Pontryagin 
dual of $\Sel_{\mathfrak{p}^{\infty}}(A/K_{\mathfrak{p}^{\infty}})$, with $X(A/K_{\mathfrak{p}^{\infty}})$ its $\Lambda$-cotorsion submodule. 
Given an element $\lambda \in \Lambda$, let $\lambda^*$ denote the image of $\lambda$ under the involution of $\Lambda$ induced by 
inversion in $G=G_{\mathfrak{p}^{\infty}}$. The main conjecture posed by Howard \cite[Theorem B]{Ho2}, following Perrin-Riou \cite{PR87}, 
can then be summarized as in the following way. 

\begin{conjecture}[Iwasawa main conjecture] If $A$ has ordinary reduction at $\mathfrak{p}$, then

\begin{itemize}
\item[(i)] The $\Lambda$-modules $\mathfrak{H}_{\infty}$ and $\mathfrak{S}_{\infty}$ are torsionfree of rank one.
\item[(ii)] The dual Selmer group $X(A/K_{\mathfrak{p}^{\infty}})$ has rank one as a $\Lambda$-module.
\item[(iii)] There is a pseudoisomorphism of $\Lambda$-modules 
\begin{align*}X(A/K_{\mathfrak{p}^{\infty}})_{\tors} \longrightarrow M \oplus M \oplus M\mathfrak{p}.
\end{align*} Here, the $\Lambda$-characteristic power series $\operatorname{char}(M)$ is prime to 
$\mathfrak{p}\Lambda$ with the property that $\operatorname{char}(M) = \operatorname{char}(M)^*$, 
and $\operatorname{char}(M\mathfrak{p})$ is a power of $\mathfrak{p}\Lambda$.
\item[(iv)] The following equality of ideals holds in $\Lambda$, 
at least up to powers of $\mathfrak{p}\Lambda$:
\begin{align*}\left( \operatorname{char}(M) \right) = 
\left( \operatorname{char}\left(\mathfrak{S}_{\infty}/ \mathfrak{H}_{\infty} \right) \right).\end{align*} 
\end{itemize} \end{conjecture} 

\begin{remark} Note that many cases of this conjecture (with one divisibility in (iv)) are established by Howard \cite[Theorem B]{Ho2}, 
using the relevant nonvanishing theorem of Cornut-Vatsal \cite{CV} to ensure nontriviality. \end{remark}

Let us now re-state this conjecture in terms of Euler characteristic formulae. Let  $\Sel_{\mathfrak{p}^{\infty}}(A/K_{\mathfrak{p}^{\infty}})_{\cotors}$ 
denote the $\Lambda$-cotorsion submodule of $\Sel_{\mathfrak{p}^{\infty}}(A/K_{\mathfrak{p}^{\infty}})$. 
Recall that given an element $\lambda \in \Lambda$, we write $\lambda(0)$ to denote the image of $\lambda$ under the natural map 
$\Lambda \longrightarrow {\bf{Z}}_p$. Recall as well that we let $\mathfrak{K}_n$ denote the Artin symbol of $\mathfrak{d}_n = (\sqrt{D}\mathcal{O}_K) \cap \mathcal{O}_{p^n}$. 
Here, $D$ denotes the absolute discriminant of $K$, and $\mathcal{O}_{p^n}$ the 
$\mathcal{O}_F$-order of conductor $\mathfrak{p}^n$ in $K$, i.e. 
$\mathcal{O}_{p^n}= \mathcal{O}_F + \mathfrak{p}^n \mathcal{O}_K$. 
Let $\mathfrak{K} = \varprojlim_n \mathfrak{K}_n$. 
We let $\langle ~,~\rangle_{A, K[p^n]}$ denote the $p$-adic height pairing 
\begin{align*} \langle ~,~\rangle_{A, K[p^n]}: A^t(K[p^n]) \times A(K[p^n]) 
&\longrightarrow {\bf{Q}}_p, \end{align*} defined e.g. in \cite[(9), $\S 3.3$]{Ho} 
or \cite{PR91}. Assume for simplicity that $A$ is principally polarized. 
There exists by a construction of Perrin-Riou \cite{PR87} (see also \cite{PR91}) 
a $p$-adic height pairing \begin{align*} \mathfrak{h}_n: \mathfrak{S}(A/K_{p^n}) 
\times \mathfrak{S}(A/K_{p^n}) &\longrightarrow c^{-1}{\bf{Z}}_p \end{align*} 
whose restriction to the image of the Kummer map $A(K[p^n]) \otimes {\bf{Z}}_p 
\rightarrow \mathfrak{S}(E/K[p^n])$ coincides with the pairing $\langle ~,~\rangle_{A, K[p^n]}$ 
after identifying $A^t \cong A$ in the canonical way (cf. \cite[Proposition 0.0.4]{Ho}). 
Here, $c \in {\bf{Z}}_p$ is some integer that does not depend on the choice of $n$. 
Following \cite{PR87} and \cite{Ho}, these pairings can be used to construct a pairing 
\begin{align*} \mathfrak{h}_{\infty}: \mathfrak{S}_{\infty} \times \mathfrak{S}_{\infty} 
&\longrightarrow c^{-1}{\bf{Z}}_p[[G]] \\ (\varprojlim a_n, \varprojlim_n b_n) 
&\longmapsto \varprojlim_n \sum_{\sigma \in G_{p^n}} \mathfrak{h}_n(a_n, b_n^{\sigma}) 
\cdot \sigma. \end{align*} We then define the {\it{$p$-adic regulator $\mathcal{R} = 
\mathcal{R}(A/K_{p^{\infty}})$}} to be the image of $\mathfrak{S}_{\infty}$ in $c^{-1}{\bf{Z}}_p[[G]]$ under this pairing 
$\mathfrak{h}_{\infty}$. Let ${\bf{e}}$ denote the natural projection \begin{align*} {\bf{e}}: {\bf{Z}}_p[[G[p^{\infty}]]] 
&\longrightarrow {\bf{Z}}_p[[G]]. \end{align*} We then make the following conjecture in this situation.

\begin{conjecture}\label{-1} 
Let $A$ be a principally polarized abelian variety of arithmetic 
conductor $N$ defined over $F$ having good ordinary reduction 
at each prime above $p$ in $K$. Assume as well that $A$ satisfies 
the weak-Heegner hypothesis with respect to $N$ and $K$, i.e. that 
$\eta(N) = (-1)^{d-1}$. Then, the $G$-Euler characteristic 
$\chi(G, \Sel(A/K_{\mathfrak{p}^{\infty}})_{\cotors})$ is 
well-defined, and given by the expression 
\begin{align*} \chi(G, \Sel(A/K_{\mathfrak{p}^{\infty}})_{\cotors}) &= \vert 
\operatorname{char}(\mathfrak{S}_{\infty}/\mathfrak{H}_{\infty})(0) \cdot 
\operatorname{char}(\mathfrak{S}_{\infty}/\mathfrak{H}_{\infty})^*(0) 
\cdot \mathfrak{R}(0)\vert_p^{-1}. \end{align*} Here, $\mathfrak{R}$ 
denotes the ideal $({\bf{e}}(\mathfrak{K}))^{-1} \mathcal{R}$, which 
lies in the ${\bf{Z}}_p$-Iwasawa algebra ${\bf{Z}}_p[[G]]$. Thus, we
also conjecture that the $p$-adic regulator $\mathcal{R} 
= \mathcal{R}(A/K_{p^{\infty}})$ is not identically zero, and 
hence that the $p$-adic height pairing $\mathfrak{h}_{\infty}$ 
is nondegenerate.\end{conjecture} \end{remark}

\begin{remark}[Modular elliptic curves over imaginary quadratic fields.]

Let us now give a more precise description of the conjectural formula for 
elliptic curves defined over the totally real field $F = {\bf{Q}}$. We first recall the 
$\Lambda$-adic Gross-Zagier theorem of Howard \cite[Theorem B]{Ho}.
Let $p$ be an odd prime. Fix embeddings $\overline{\bf{Q}} \rightarrow \overline{\bf{Q}}_p$ 
and $\overline{\bf{Q}} \rightarrow {\bf{C}}$. Let $E$ be an elliptic curve of conductor $N$ defined 
over ${\bf{Q}}$. Hence, $E$ is modular by fundamental work of Wiles, Taylor-Wiles, and 
Breuil-Conrad-Diamond-Taylor. Let $f \in S_2(\Gamma_0(N))$ denote the cuspidal newform 
of weight $2$ and level $N$ associated to $E$ by modularity. Let $K$ be an imaginary quadratic 
field of discriminant $D$, and associated quadratic character $\eta$. We impose the following

\begin{hypothesis}\label{hyp}
Assume (i) that $f$ is {\it{$p$-ordinary}} in the sense that the the image of its $T_p$-eigenvalue under the fixed embedding 
$\overline{{\bf{Q}}} \rightarrow \overline{\bf{Q}}_p$ is a $p$-adic unit, (ii) that the {\it{Heegner hypothesis}} holds with 
respect to $N$ and $K$, i.e. that each prime divisor of $N$ split in $K$, and (iii) that $(p, ND)=1$. 
\end{hypothesis} Let us for clarity also fix the following notations here. 
We write $K[p^{\infty}]$ to denote the $p^{\infty}$-ring class tower over $K$, 
with Galois group $\varOmega=\Gal(K[p^{\infty}]/K)$. We then write $K_{p^{\infty}}$ to denote the dihedral or anticyclotomic 
${\bf{Z}}_p$-extension of $K$, with Galois group $\Omega = \Gal(K_{p^{\infty}}/K) \approx {\bf{Z}}_p$. 
We also write $K(\mu_{p^{\infty}})$ to denote the extension of $K$ obtained 
by adjoining all primitive, $p$-th power roots of unity, with Galois group $\varGamma = \Gal(K(\mu_{p^{\infty}})/K)$. 
We then write $K^{\cyc}$ to denote the cyclotomic ${\bf{Z}}_p$-extension of $K$, with Galois group $\Gamma = 
\Gal(K^{\cyc}/K) \approx {\bf{Z}}_p$. We shall consider the compositum extension 
$R_{\infty} = K[p^{\infty}] \cdot K(\mu_{p^{\infty}})$ with Galois group 
$\mathcal{G} = \Gal(R_{\infty}/K) = \varOmega \times \varGamma$. We 
shall also also consider the ${\bf{Z}}_p^2$-extension $K_{\infty}$ of $K$
contained in $R_{\infty}$, with Galois group $G = \Gal(K_{\infty}/K) = 
\Omega \times \Gamma \approx {\bf{Z}}_p^2$. 
The construction of Hida \cite{Hi} and Perrin-Riou \cite{PR88} (whose integrality is shown 
in \cite[Theorem 2.9]{VO3}) gives a two-variable $p$-adic $L$-function 
$$\mathcal{L}_f \in {\bf{Z}}_p[[\mathcal{G}]] = {\bf{Z}}_p[[ \varOmega \times \varGamma]].$$ 
This element interpolates the algebraic values $$L(E/K, \mathcal{W}, 1)/8\pi^2\langle f, f \rangle_N 
= L(f \times g_{\mathcal{W}}, 1)/8 \pi^2 \langle f, f \rangle,$$ which we view as elements of 
$\overline{\bf{Q}}_p$ under our fixed embedding $\overline{\bf{Q}} \longrightarrow \overline{\bf{Q}}_p$.
Here, $\mathcal{W}$ is any finite order character of $\mathcal{G}$, and $\langle f, f \rangle_N$ the 
Petersson inner product of $f$ with itself. Let us now commit a minor abuse of notation in also
writing $\mathcal{L}_f$ to denote the image of this $p$-adic $L$-function in the Iwasawa algebra
${\bf{Z}}_p[[\varOmega \times \Gamma]]$. Fixing a topological generator $\gamma$ of $\Gamma$, 
we can then write this two-variable $p$-adic $L$-function $\mathcal{L}_f$ as a power series in the 
cyclotomic variable $(\gamma-1)$, \begin{align}\label{exp} \mathcal{L}_f &= \mathcal{L}_{f, 0} 
+ \mathcal{L}_{f, 1}(\gamma -1) + \mathcal{L}_{f,2}(\gamma -1)^2 + \ldots \end{align} Here, the 
coefficients $\mathcal{L}_{f, i}$ are elements of the completed group ring ${\bf{Z}}_p[[\varOmega]]$. 

Now, since we assume the Heegner hypothesis in Hypothesis \ref{hyp} (ii), it is a well known 
consequence that the functional equation for $L(E/K, s) = L(f \times g_K, s)$ forces the central
value $L(E/K, 1) = L(f \times g_K, 1)$ to vanish. In fact, it is also well known that $L(E/K, \mathcal{W}, 1) 
= L(f \times g_{\mathcal{W}}, 1)$ vanishes for $\mathcal{W}$ any finite order character
of $\varOmega$ in this situation (see e.g. \cite[$\S 1$]{CV}). Hence, in this situation, 
it follows from the interpolation property of $\mathcal{L}_f$ that we must have $\mathcal{L}_{f,0}=0$. 
The main result of Howard \cite{Ho} shows that the linear term $\mathcal{L}_{f,1}$ in this setting is related 
to height pairings of Heegner points in the Jacobian $J_0(N)$ of the modular curve $X_0(N)$. To be more 
precise, the Heegner hypothesis implies for each integer $n \geq 0$ the existence of a family of Heegner 
points $h_n \in X_0(N)({\bf{C}})$ of conductor $p^n$. These Heegner points $h_n$ are defined in the usual way: 
each $h_n$ is a cyclic $N$-isogeny $h_n: \mathcal{E}_n \rightarrow \mathcal{E}_n'$ of elliptic curves 
$\mathcal{E}_n$ and $\mathcal{E}_n'$ having exact CM by the ${\bf{Z}}$-order $\mathcal{O}_{p^n}= 
{\bf{Z}} + p^n \mathcal{O}_K$. As explained in \cite{Ho}, the family $\lbrace h_n \rbrace_n$ can be 
chosen in such a way that the following diagram commutes, where the vertical arrows are $p$-isogenies: 
\begin{align*} \begin{CD}\mathcal{E}_n @>{h_n}>> \mathcal{E}_{n}'  \\ @VVV @VVV \\  \mathcal{E}_{n-1} 
@>{h_{n-1}}>> \mathcal{E}_{n-1}'. \end{CD}\end{align*} This in particular makes $\mathcal{E}_{n-1}$ 
a quotient of $\mathcal{E}_n$ by its $p\mathcal{O}_{p^{n-1}}$-torsion, and $\mathcal{E}_{n-1}'$ 
a quotient of $\mathcal{E}_{n}'$ by its $p\mathcal{O}_{p^{n-1}}$-torsion. By the theory of complex 
multiplication, the CM elliptic curves $\mathcal{E}_n$ and $\mathcal{E}_n'$, as well as the isogeny 
$h_n$ between them, can all be defined over the ring class field $K[p^n]$ of conductor $p^n$ over $K$. 
Hence, we have for each integer $n \geq 0$ a family of Heegner points $h_n \in X_0(N)(K[p^n])$. 
Let us commit an abuse of notation in also writing $h_n$ to denote the image of any such point 
under the usual embedding $X_0(N) \rightarrow J_0(N)$ that sends the cusp at infinity to the origin.

Let ${\bf{T}}$ denote the ${\bf{Q}}$-algebra generated by the Hecke operators $T_l$ at primes $l$ 
not dividing $N$ acting on $J_0(N)$. The semisimplicity of ${\bf{T}}$ induces ${\bf{T}}$-module 
isomorphism \begin{align*} J_0(N)(K[p^n]) \longrightarrow \bigoplus_{\beta} J(K[p^n])_{\beta}. 
\end{align*} Here, the direct sum runs over all $\Gal(\overline{\bf{Q}}_p/{\bf{Q}}_p)$-orbits of 
algebra homomorphisms $\beta: {\bf{T}} \rightarrow \overline{\bf{Q}}_p$, and each summand 
is fixed by the natural action of $\Gal(K[p^n]/{\bf{Q}})$. If $\beta({\bf{T}})$ is contained in ${\bf{Q}}_p$, 
then the Hecke algebra ${\bf{T}}$ acts on the summand $J(K[p^n])_{\beta}$ via the homomorphism 
$\beta$. Our fixed newform $f \in S_2(\Gamma_0(N))$ determines such a homomorphism, 
which we shall also denote by $f$ in an abuse of notation. We shall then write $h_{n, f}$ 
to denote the image of a Heegner point $h_n \in J_0(N)(K[p^n])$ in the associated summand 
$J(K[p^n])_f$. Let $\alpha_p$ denote the $p$-adic unit root of the Hecke polynomial $X^2 -a_p(f)X+p$, 
where $a_p(f)$ denotes the $T_p$-eigenvalue of $f$. Using the language of \cite{BD96}, 
we define for each integer $n\geq 0$ a sequence of 
{\it{regularized Heegner points}}: \begin{align*} z_n &= 
\frac{1}{\alpha_p^n} h_{n, f} - \frac{1}{\alpha_p^{n-1}} h_{n-1, f} \end{align*} 
if $n \geq 1$, otherwise \begin{align*} z_0 &=  \frac{1}{u} \cdot 
\begin{cases} \left(1 - \frac{\sigma_p}{\sigma_p^*} \right)
\left(1- \frac{\sigma_p^*}{\sigma_p} \right) h_{0,f} &\text{if $\eta(p)=1$} 
\\  \left(1 - \frac{1}{\alpha_p^2} \right) h_{0,f} &\text{if $\eta(p)=-1.$} 
\end{cases}\end{align*} Here, $u = \frac{1}{2}\vert \mathcal{O}_K^{\times}\vert$,
and $\sigma_p$ and $\sigma_p^*$ denote the Frobenius automorphisms in 
the Galois group $\Gal(K[1]/K) \approx \Pic(\mathcal{O}_K)$ of the Hilbert class 
field $K[1]$ over $K$ at the two primes above $p$ in the case where $\eta(p)=1$. 
It can be deduced from the Euler system relations shown in Howard 
\cite{Ho} that these regularized Heegner points are compatible under 
norm and trace maps on the summand $J(K[p^n])_f$. We refer the 
reader to \cite[$\S$ 1.2]{Ho} for details. 

The following main result of Howard \cite{Ho} was first proved by 
Perrin-Riou \cite{PR87} for the case of $n=0$. To state this result, 
let $\mathfrak{K}_n$ denote the Artin symbol of 
$\mathfrak{d}_n = (\sqrt{D}\mathcal{O}_K) \cap \mathcal{O}_{p^n}$. 
Let $\langle~,~\rangle_{E, K[p^n]}$ denote the canonical $p$-adic height 
pairing \begin{align*} \langle ~ , ~\rangle_{E, K[p^n]}: E^t(K[p^n]) \times E(K[p^n]) 
\longrightarrow {\bf{Q}}_p, \end{align*} as defined in \cite[(9), $\S$ 3.3]{Ho}. 
Note that this pairing is canonical as a consequence 
of the fact that $E$ is ordinary at $p$, along with the uniqueness claims of 
\cite[Proposition 3.2.1]{Ho}. The reader will also note that we do not extend 
this pairing ${\bf{Q}}_p$-linearly as done throughout \cite{Ho}, 
to obtain the results stated here after taking tensor products 
$\otimes {\bf{Q}}_p$. Indeed, such extensions are not necessary 
as $\mathcal{L}_f$ is integral by \cite[Theorem 2.9]{VO3}, and hence 
belongs to the Iwasawa algebra ${\bf{Z}}_p[[\mathcal{G}]]$ (as opposed 
to just ${\bf{Z}}_p[[\mathcal{G}]] \otimes_{{\bf{Z}}_p} {\bf{Q}}_p$). 
Let $\log_p$ denote the $p$-adic logarithm, composed with a fixed isomorphism 
$\Gamma \cong {\bf{Z}}_p^{\times}$. Let us also write $z_n^t$ denote the image 
of the regularized Heegner point $z_n$ under the canonical principal polarization 
$J_0(N) \cong J_0(N)^t$. We deduce from \cite[Theorem A]{Ho} the following result.

\begin{theorem}[Howard]\label{HAI} Assume that $D$ is odd and not equal to $-3$. 
Assume as well that $\eta(p)=1$. Then, for any dihedral character 
$\rho: \Gal(K[p^n]/K) \longrightarrow \overline{\bf{Q}}_p^{\times}$, we have the 
identity \begin{align}\label{HA} \rho(\mathfrak{K}_n) \cdot \log_p(\gamma) \cdot \rho(\mathcal{L}_{f,1}) 
&= \sum_{\sigma \in \Gal(K[p^n]/K)} \rho(\sigma) \langle z_n^t, z_n^{\sigma} \rangle_{E, K[p^n]}. \end{align}
This identity $(\ref{HA})$ is independent of the choice of topological generator $\gamma \in \Gamma$. 
\end{theorem} 

This result has the following implications for our modular elliptic curve $E$, which belongs 
the isogeny class of ordinary elliptic curves associated to the newform $f$. Fix a modular 
parametrization $\varphi: X_0(N) \rightarrow E$. Let $\varphi_*: J_0(N) \rightarrow E$ 
denote the induced Albanese map, and $\varphi^*:E  \rightarrow J_0(N)$ the induced 
Picard map. Let \begin{align*} x_n &= \varphi_*(z_n) \in E(K[p^n])\end{align*} 
denote the image of the regularized Heegner point $z_n$ under the Albanese map 
$\varphi_*$, and let $x_n^t$ denote the unique point of $E(K[p^n])$ for which 
$\varphi^*(x_n^t) = z_n^t$. Hence, we can identify $x_n$ with $\deg(\varphi) x_n^t$. 
Note that the points $x_n$ and $x_n^t$ are norm compatible 
as $n$ varies, thanks to the norm compatibility of the regularized Heegner points $z_n$ and 
$z_n^t$. Granted this property, we define the {\it{Heegner $p$-adic $L$-function}} 
$\mathcal{L}_{\Heeg}$ in  ${\bf{Z}}_p[[\varOmega]]$ by 
\begin{align*} \mathcal{L}_{\Heeg} &= \varprojlim_n \sum_{\sigma \in \Gal(K[p^n]/K)} 
\langle x_n^t, x_n^{\sigma} \rangle_{E, K[p^n]} \cdot \sigma. \end{align*} We deduce 
from Theorem \ref{HAI} the following integral version of \cite[Theorem B]{Ho}.

\begin{theorem}\label{HBI} Keep the setup and hypotheses of Theorem \ref{HAI} above. 
Let us write $\mathfrak{K} = \varprojlim_n \mathfrak{K}_n \in \varOmega$. Then, as ideals 
in the Iwasawa algebra ${\bf{Z}}_p[[\varOmega]]$, we have \begin{align*} \mathfrak{K} \cdot 
\log_p(\gamma) \cdot \mathcal{L}_{f,1} &= \mathcal{L}_{\Heeg}. \end{align*} \end{theorem} 
This result in turn has applications to the $\Lambda$-adic Gross-Zagier theorem of Mazur-Rubin \cite[Conjecture 9]{MR}, 
where $\Lambda$ denotes the Iwasawa algebra ${\bf{Z}}_p[[\Omega]]$. Following \cite[$\S$ 0]{Ho}, we define a pairing 
\begin{align*} \langle ~,~\rangle_{E, K[p^n]}^{\Gamma}: E(K[p^n])\times E(K[p^n]) &\longrightarrow \Gamma \end{align*} 
by the implicit relation $\langle ~,~\rangle_{E, K[p^n]} = \log_p \circ \langle ~,~\rangle_{E, K[p^n]}^{\Gamma}$. 
Here, we should mention that the construction of $\langle ~,~\rangle_{E, K[p^n]}$ given in \cite[$\S$ 3.3]{Ho} 
requires as input an auxiliary (idele class) character 
$$\begin{CD}\rho_{K[p^n]}: {\bf{A}}_{K[p^n]}^{\times}/K[p^n]^{\times} @>>> \Gamma @>{\log_p}>> {\bf{Z}}_p. \end{CD}$$ 
Anyhow, we then define from this a completed group ring element 
\begin{align*} \mathcal{L}_{\Heeg}^{\Gamma} &= \varprojlim_n \sum_{\sigma \in \Gal(K[p^n]/K)} \langle x_n, x_n^{\sigma} \rangle_{E, K[p^n]}^{\Gamma} 
\cdot \sigma \in {\bf{Z}}_p[[\Omega]] \otimes \Gamma. \end{align*} Let $I$ denote the kernel of the natural projection 
\begin{align*} {\bf{Z}}_p[[ \Omega \times \Gamma]] &\longrightarrow {\bf{Z}}_p[[\Omega]]. \end{align*} Let $\vartheta$ denote the isomorphism 
\begin{align*} \vartheta: {\bf{Z}}_p[[\Omega]] \otimes \Gamma &\longrightarrow  I/I^2 \\ (\lambda \otimes \gamma) &\longmapsto  \lambda(\gamma-1). \end{align*} 
As explained in \cite{Ho}, we can derive from this the formula \begin{align*} \vartheta \left( \mathcal{L}_{\Heeg}^{\Gamma}\right) 
&= \frac{\deg(\varphi)}{\log_p(\gamma)} \cdot \mathcal{L}_{\Heeg}(\gamma-1).\end{align*} 
Now, observe that since $\mathcal{L}_{f, 0}=0$ with our hypothesis (i.e. Hypothesis \ref{hyp} (ii)), 
the two-variable $p$-adic $L$-function $\mathcal{L}_f$ is contained in the kernel $I$. Hence, 
Theorem \ref{HBI} gives us the following identifications in the quotient $I/I^2$: 
\begin{align}\label{I/I^2} \mathcal{L}_f &= \mathcal{L}_{f, 1}(\gamma-1) 
= \frac{\mathfrak{L}_{\Heeg}}{\mathfrak{K} \cdot \log_p(\gamma)} \cdot (\gamma -1) 
= \frac{\vartheta \left( \mathcal{L}_{\Heeg}^{\Gamma}\right)}
{\mathfrak{K} \cdot \deg(\varphi)} .\end{align} Let us now keep the hypotheses of Theorem \ref{HAI} in force. 
We assume additionally that the absolute Galois group $G_K = \Gal(\overline{K}/K)$ 
of $K$ surjects onto the ${\bf{Z}}_p$-automorphisms of the $p$-adic Tate module $T_p(E)$ 
of $E$, and that $p$ does not divide the class number of $K$. 
Let us now write $\widetilde{x}_{\infty} \in \mathfrak{S}(E/K_{p^{\infty}})$ 
to denote the projective limit of the sequence of norm compatible elements 
\begin{align*} \widetilde{x}_n &= \Norm_{K[p^{n+1}]/K_{p^n}}(x_{n+1}) \in 
\mathfrak{S}(E/K_{p^n}). \end{align*} We again write $\mathfrak{H}_{\infty} 
= \mathfrak{H}_{\infty}(E/K_{p^{\infty}})$ to denote the Heegner or CM submodule 
of the compactified Selmer group $\mathfrak{S}_{\infty} = \mathfrak{S}(E/K_{p^{\infty}})$ 
generated by this $\widetilde{x}_{\infty}$. By the main theorem of Cornut \cite{Cor}
(see also \cite{CV}), $\mathfrak{H}_{\infty}$ has rank one as 
a $\Lambda$-module. By work of Howard \cite{Ho0} (see also 
\cite{Ho}) and Bertolini \cite{Be}, we also know (i) that  
$X(E/K_{p^{\infty}})$ is a rank one $\Lambda$-module, (ii) that 
$\mathfrak{S}_{\infty}= \mathfrak{S}(E/K_{p^{\infty}})$ is free of 
rank one over $\Lambda$ (whence the quotient 
$\mathfrak{S}_{\infty}/\mathfrak{H}_{\infty}$ is a torsion 
$\Lambda$-module), and (iii) that the following divisibility 
of ideals holds in $\Lambda$: \begin{align}\label{-1div}  
\operatorname{char}X(E/K_{p^{\infty}})_{\tors}  \mid  
\operatorname{char} (\mathfrak{S}_{\infty}/\mathfrak{H}_{\infty}) \cdot 
\operatorname{char}(\mathfrak{S}_{\infty}/\mathfrak{H}_{\infty})^* . \end{align} 
Here again, we have written $X(E/K_{p^{\infty}})_{\tors}$ to denote the 
$\Lambda$-torsion submodule of $X(E/K_{p^{\infty}})$. Let us now 
consider the following result of Perrin-Riou \cite{PR87}, as 
described in \cite[Proposition 0.0.4]{Ho}.

\begin{theorem}[Perrin-Riou]\label{IPR} There exists a $p$-adic height 
pairing \begin{align}\label{PRP} \mathfrak{h}_n: \mathfrak{S}(E/K_{p^{n}}) 
\times \mathfrak{S}(E/K_{p^{n}}) &\longrightarrow c^{-1}{\bf{Z}}_p 
\end{align} whose restriction to the image of the Kummer map 
$E(K[p^n]) \otimes {\bf{Z}}_p \rightarrow \mathfrak{S}(E/K[p^n])$ 
coincides with the pairing $\langle~,~\rangle_{E, K[p^n]}$. Here,
 $c \in {\bf{Z}}_p$ is some $p$-adic integer that does not depend 
 on choice of $n$. \end{theorem} Using this pairing $(\ref{PRP})$, 
 we can then define a pairing \begin{align*} \mathfrak{h}_{\infty}: 
 \mathfrak{S}_{\infty} \times \mathfrak{S}_{\infty} &\longrightarrow 
 c^{-1}{\bf{Z}}_p[[\Omega]] \\ (\varprojlim_n a_n, \varprojlim_n b_n) 
 &\longmapsto \varprojlim_n \sum_{\sigma \in \Gal(K_{p^n}/K) } 
 \mathfrak{h}_n (a_n, b_n^{\sigma}) \cdot \sigma. \end{align*} 
 We again define the $p$-adic regulator $\mathcal{R} = \mathcal{R}(E/K_{p^{\infty}})$ 
 of $E$ over $K_{p^{\infty}}$ to be the image in $c^{-1}{\bf{Z}}_p[[\Omega]]]$ 
 of this pairing $\mathfrak{h}_{\infty}$. Let ${\bf{e}}$ denote the natural 
projection \begin{align*} {\bf{e}}: {\bf{Z}}_p [[\varOmega]]
\longrightarrow {\bf{Z}}_p[[\Omega]]. \end{align*}  Following Howard 
\cite{Ho} (using \cite[Remark 3.2.2]{Ho} with the uniqueness 
claims of \cite[Proposition 3.2.1]{Ho} and the fact that $E$ is ordinary 
at $p$), we deduce that the height pairing $\mathfrak{h}_{\infty}$ is 
norm compatible. This allows us to deduce that 
\begin{align}\label{crux} {\bf{e}}\left( \mathcal{L}_{\Heeg}\right) 
&= \mathfrak{h}_{\infty}(\widetilde{x}_{\infty}, \widetilde{x}_{\infty}) =
\operatorname{char}\left( \mathfrak{S}_{\infty}/H_{\infty}\right) \cdot 
\operatorname{char}\left( \mathfrak{S}_{\infty}/H_{\infty}\right)^* 
\mathcal{R} \end{align} as ideals in $c^{-1}{\bf{Z}}_p[[\Omega]]$.
On the other hand, we obtain from $(\ref{I/I^2})$ 
the identifications \begin{align*} {\bf{e}}(\mathcal{L}_{f,1}) 
&= \frac{\mathfrak{h}_{\infty}(\widetilde{x}_{\infty}, 
\widetilde{x}_{\infty})}{{\bf{e}}(\mathfrak{K}\cdot \log_p(\gamma))} 
= \frac{\operatorname{char}\left( \mathfrak{S}_{\infty}/H_{\infty}\right) 
\cdot \operatorname{char}\left( \mathfrak{S}_{\infty}/H_{\infty}\right)^* 
\mathcal{R}}{{\bf{e}}(\mathfrak{K} \cdot \log_p(\gamma))} \end{align*} 
as ideals in $\Lambda = {\bf{Z}}_p[[\Omega]]$. Putting this all together,
we obtain the following result.

\begin{corollary}\label{EC-1} Assume that $D$ is odd and not equal to $-3$. 
Assume additionally that the absolute Galois group $G_K = \Gal(\overline{K}/K)$ 
surjects onto the ${\bf{Z}}_p$-automorphisms of $T_p(E)$, that $\eta(p)=1$, and 
that $p$ does not divide the class number of $K$. Then, we have the following 
formal relations: \begin{align*} \chi(G, \Sel_{p^{\infty}}(A/K_{p^{\infty}})_{\cotors}) 
&\geq  \vert \operatorname{char}(\mathfrak{S}_{\infty}/\mathfrak{H}_{\infty})(0) 
\cdot \operatorname{char}(\mathfrak{S}_{\infty}/\mathfrak{H}_{\infty})^*(0) \vert_p^{-1} \\
&\geq  \vert \operatorname{char}(\mathfrak{S}_{\infty}/\mathfrak{H}_{\infty})(0) 
\cdot \operatorname{char}(\mathfrak{S}_{\infty}/\mathfrak{H}_{\infty})^*(0) 
\mathcal{R}(0) \vert_p^{-1} \\ &= \vert {\bf{e}}(\mathcal{L}_{f,1}(0)) \vert_p^{-1} \\ 
&= \vert \mathfrak{h}_{\infty}(\widetilde{x}_{\infty}, \widetilde{x}_{\infty})(0) \vert_p^{-1} \end{align*} \end{corollary} 

\begin{proof} Note that ${\bf{e}}(\mathfrak{K} \cdot \log_p(\gamma))$ is a unit, and 
so we can ignore contributions from the obvious ${\bf{e}}(\mathfrak{K} \cdot \log_p(\gamma))^{-1} $
 terms. The first inequality follows from the main conjecture divisibility shown in \cite{Ho2}.
The second inequality is trivial, at least granted the conjectural nontriviality of $\mathcal{R}$. 
The third and fourth equalities follow from the Iwasawa theoretic Gross-Zagier theorem shown 
in \cite{Ho}, as described above. \end{proof}

Note as well that this result is formal, as we have not assumed the $G$-Euler characteristic 
to be well-defined, i.e. as we have not assumed that the $p$-adic regulator $\mathcal{R}$ is
nondegenerate. We therefore end this section with the following speculation.

\begin{conjecture} Assume the conditions of Hypothesis \ref{hyp} are true. 
Recall that we let $\Lambda$ denote the Iwasawa algebra 
$\Lambda(G) = {\bf{Z}}_p[[G]]$. Then, (i) the $G$-Euler characteristic 
$\chi(G, \Sel_{p^{\infty}}(A/K_{p^{\infty}}))$ is well-defined, and given by 
the formula \begin{align*}  \chi(G, \Sel_{p^{\infty}}(A/K_{p^{\infty}})_{\cotors}) &= 
 \vert \operatorname{char}(\mathfrak{S}_{\infty}/\mathfrak{H}_{\infty})(0) 
\cdot \operatorname{char}(\mathfrak{S}_{\infty}/\mathfrak{H}_{\infty})^*(0) 
\mathcal{R}(0) \vert_p^{-1}\\  &= \vert {\bf{e}}(\mathcal{L}_{f,1}(0)) \vert_p^{-1} \\
&= \vert \mathfrak{h}_{\infty}(\widetilde{x}_{\infty}, \widetilde{x}_{\infty})(0) \vert_p^{-1} 
\end{align*} Moreover, (ii) the image in the Iwasawa algebra $\Lambda$ of  the two-variable 
$p$-adic $L$-function $\mathcal{L}_f = \mathcal{L}_{f,1}(\gamma -1)$ is not identically zero, 
(iii) the $\Lambda$-adic height pairing $\mathfrak{h}_{\infty}$ of Perrin-Riou is nondegenerate, 
and (iv) the $p$-adic regulator $\mathcal{R} = \mathcal{R}(E/K_{p^{\infty}})$ of $E$ over $K_{p^{\infty}}$ 
is nondegenerate, and satisfies the relation $\mathcal{R}={\bf{e}}(\mathfrak{K} \cdot \log_p(\gamma))^{-1} 
\Lambda$. \end{conjecture} 

\begin{remark} Note that the nonvanishing condition for $\mathcal{L}_f$ presented in 
part (i) of this conjecture does not follow a priori from the nonvanishing theorems of Cornut \cite{Cor} or 
Cornut-Vatsal \cite{CV}. Those results show here that the complex values $L'(E/K, \rho, 1)$ do not vanish, 
where $\rho$ a ring class character of $G[p^{\infty}]$ of sufficiently large conductor. In particular, what is 
required is a strengthening of their result to include twists by characters of the cyclotomic Galois group 
$\Gamma = \Gal(K(\mu_{p^{\infty}})/K)$. \end{remark}

\end{remark}

\begin{remark}[Acknowledgements.] The author is grateful to John Coates, Christophe Cornut, Olivier Fouquet, 
Ben Howard, and Dimitar Jetchev for useful discussions, as well as to an anonymous referee for pointing out
a bad error in a earlier version of this work (which has since been corrected). \end{remark}

\end{document}